\documentclass{article}
\usepackage{graphicx} % Required for inserting images
\usepackage{amsmath,amssymb,amsthm,textcomp, tabu,appendix}
\usepackage{xcolor}
\usepackage{youngtab}
\usepackage{natbib}
\usepackage{tikz}
\usetikzlibrary{positioning, arrows.meta, calc}
\usepackage{geometry}
\usepackage{mathtools}
\usepackage{hyperref} 
\definecolor{boxblue}{RGB}{173, 216, 230}   % light blue fill
\definecolor{boxgray}{RGB}{200, 190, 230}   % light gray fill
\definecolor{refcolor}{RGB}{0, 128, 128}    % teal for cross-references

\numberwithin{equation}{section}
\newcommand{\boxt}[1]{]_{#1}}

\newcommand{\mb}[1]{\mathbb{#1}}
\newcommand{\mf}[1]{\mathfrak{#1}}

\newcommand{\abs}[1]{\lvert#1\rvert}

\newcommand{\C}{\mathbb{C}}

\newcommand{\m}{\mathfrak{m}}

\newtheorem{theorem}{Theorem}[section]
\newtheorem{conjecture}[theorem]{Conjecture}
\newtheorem{lemma}[theorem]{Lemma}

\newtheorem{maintheorem}{Theorem}

\newtheorem{maincorollary}{Corollary}
\newtheorem{proposition}[theorem]{Proposition}

\theoremstyle{definition}

\newtheorem{definition}{Definition}[section]
\theoremstyle{theorem}
\newtheorem{remark}[definition]{Remark}
\title{A Combinatorial Approach to the CUE Joint Moment Problem}
\author{George Snape}
\date{September 30, 2026}

\begin{document}

\maketitle
\begin{abstract}
 In this paper, we give exact combinatorial expressions for all CUE joint moments on the unit circle, in the case that all exponents are even positive integers, yielding the corresponding asymptotics. 
 
 We offer two equivalent expressions for these quantities: the first is a sum involving Kostka numbers; the second is related to counting integer matrices (magic squares).
 
 As an application of our results, we unify many results in the literature, which were previously disconnected. Moreover, we prove a conjecture of Basor et al. \cite{kn:basor_et_al18} regarding an irregular conformal block structure.
\end{abstract}

\section{Introduction}

In this paper, we calculate joint moments of the characteristic
polynomials, along with their derivatives, of matrices drawn randomly from
$\mathrm{U}(N)$, with respect to the Haar measure. This random matrix ensemble is referred to in the literature as the \emph{circular unitary ensemble} (CUE). For a matrix $U\in \mathrm{U}(N)$, we define the characteristic polynomial of $U$ as
\[\Lambda_U(s)=\det(I-sU^\dagger).\] We also consider the $Z$\emph{-polynomial} associated with the characteristic polynomial of a unitary matrix:
\begin{equation}\label{Z-polynomial}
  \mathcal{Z}_U(\theta)\coloneq e^{- \mf{i}N(\theta+\pi)/2}e^{\mf{i}\sum_{n=1}^N\theta_n/2}\Lambda_U(e^{\mf{i}\theta}),  
\end{equation}
where we set 
$\mf{i}^2=-1$ to avoid confusion with the index $i$, and $e^{\mf{i}\theta_1},\ldots,e^{\mf{i}\theta_N}$ are the eigenvalues of $U$.
It can be easily checked that for $\theta\in\mb{R}$, $\mathcal{Z}_U(\theta)$ is real and satisfies $\left|\mathcal{Z}_U(\theta)\right|=|\Lambda_U(e^{\mf{i}\theta})|.$

The \emph{joint moment problem} for the unitary group, as described in \cite{kn:hug01} and \cite{assiotis2025jointmomentscharacteristicpolynomials}, is the task to calculate, either exactly or asymptotically in $N$, the average 
\begin{equation}\label{eq:1}
\mathbb{E}_{\mathrm{U}(N)}\left[ \prod_{i=0}^{\ell} \bigl|\Lambda_U^{(i)}(z)\bigr|^{h_i}\right],
\end{equation}
where the expectation is taken with respect to the (normalised) Haar measure, $z\in\mathbb{C}$ and $h_0,h_1,\dots,h_\ell\in\mathbb{R}_{\ge 0}$. 

The case where the exponents in \eqref{eq:1} are even integers, and $z$ lies on the unit circle, are distinguished because the averages may be expressed in terms of the Taylor coefficients of a Schur polynomial. This property, which shall be discussed in more detail in Subsection \ref{modified moments and schur}, leads to a \emph{combinatorial approach} to the problem, which has several advantages. 

Firstly, we prove an interesting duality between Kostka numbers and counting integer matrices (see Lemma \ref{Lemma KostkaMagic}). This will unify much of the previous literature on the problem, which we discuss in Subsection \ref{unifying section}. Secondly, our treatment leads to exact expressions for finite matrix size, as well as the corresponding asymptotics. Thirdly, our combinatorial expressions
involving Kostka numbers, in the case that the joint moments are first-order derivatives, are consistent with predictions from conformal field theory for the general solution of Painlev\'e III' in terms of irregular conformal blocks (see Subsection \ref{irreg conformal block section}). This last point proves a conjecture of Basor et al. \cite{kn:basor_et_al18} (see Conjecture \ref{conj2} below).

\subsection{Main results}
In order to introduce our results, we first fix some notations regarding combinatorics of partitions and (weak) compositions.

A partition $\lambda$ is a sequence $\lambda_1\geq\lambda_2\ge\cdots$ of nonnegative integers that eventually reach zero. We write $\lambda =
(\lambda_1, \lambda_2,\ldots, \lambda_n)$ for any $n$ such that $\lambda_{n+1} = 0$ and there is a unique $n$ such that
$\lambda_n > 0$ but $\lambda_{n+1} = 0$ and this $n\coloneq \ell(\lambda)$ is the \emph{length} of $\lambda$. Each $\lambda_j$ is called a \emph{part} of $\lambda$. We call the sum of its parts $\lambda_1+\cdots +\lambda_j$ the size $ |\lambda| $ of the
partition $\lambda$. For \(M\in\mathbb{N}\), we write \(\lambda\vdash M\) if \(|\lambda|=M\), and
\(\lambda\vdash_n M\) if \(|\lambda|=M\) with \(\ell(\lambda)\le n\). 

For $i>0$, let $m_i=m_i(\lambda)$ be the number of parts $\lambda$ equal to $i$. With this notation, we write $\lambda=\langle 1^{m_1}2^{m_2}\cdots\rangle$. Thus, $\langle N^k\rangle$ is the partition with $k$ parts equal to $N$.

In this article, we define a \emph{composition} $\alpha=(\alpha,\ldots,\alpha_n)$ to be a finite sequence of positive integers. We say that $\beta$ is a \emph{weak composition} if it is a finite sequence of non-negative integers. We define its \emph{length} to be its dimension, and its size is the sum of its parts.

For $n\in\mb{N}$ and $x\in \mb{C}$, the falling factorial (or Pochammer symbol) is given by 
\[(x)_n=x(x-1)\cdots (x-n+1),\]
and if $n=0$, we set $(x)_n=1$. For a weak composition $\beta=(\beta_1,\ldots,\beta_k)$, we define
\[(x)_{\beta}=\prod_{i=1}^k(x)_{\beta_i}\quad \text{and} \quad \beta!=\prod_{i=1}^k\beta_i!.\]
For $x\in \mb{C}$, we define the \emph{generalised Pochammer symbol} to be
\[[x]_{\beta}=\prod_{i=1}^{k}\prod_{j=1}^{\beta_i}(x+j-i).\]
For two weak compositions $\alpha$ and $\beta$, we define $\alpha\cup\beta$ to be the weak composition 
\[\alpha\cup\beta=(\alpha_1,\ldots,\alpha_n,\beta_1,\ldots,\beta_m),\]
and $\alpha\pm\beta=(\alpha_1\pm\beta_1,\alpha_2\pm \beta_2,\ldots).$
When $\alpha$ and $\beta$ are partitions, we may consider $\alpha\cup\beta$ as a partition by arranging its parts in decreasing order. We write $\alpha\geq \beta$ if $\alpha_i\ge \beta_i$ for all $i$.

In this paper, we investigate the following averages. 
\begin{definition}\label{def 1}
       Let $k,h,l\in\mb{N}_0\coloneq\mb{N} \cup \{0\}$ with $h\leq 2k$ and $(h_i)_{i=1}^l$ a sequence of non-negative integers which satisfy $\sum_{i=1}^l h_i =h$. We define 
       \begin{align*}
           &\mathcal{M}_N(k;\alpha)\coloneq \int_{\mathrm{U}(N)}\left|\Lambda_U(1)\right|^{2k}\prod_{i=1}^{l}\left(\frac{\Lambda_U^{(i)}(1)}{\Lambda_U(1)}\right)^{h_i}\mathrm{d}U,\\
&|\Lambda|_N(k;\alpha)\coloneq\int_{\mathrm{U}(N)}\left|\Lambda_U(1)\right|^{2k}\prod_{i=1}^{l}\left|\frac{\Lambda_U^{(i)}(1)}{\Lambda_U(1)}\right|^{h_i}\mathrm{d}U,\\
&|\mathcal{Z}|_N(k;\alpha)\coloneq\int_{\mathrm{U}(N)}\left|\mathcal{Z}_U(0)\right|^{2k}\prod_{i=1}^{l}\left|\frac{\mathcal{Z}_U^{(i)}(0)}{\mathcal{Z}_U(0)}\right|^{h_i}\mathrm{d}U,
       \end{align*}
       where we set $\alpha$ to be the partition given by $\alpha=\langle1^{h_1}2^{h_2}\ldots l^{h_l}\rangle$ and we are integrating with respect to the Haar measure.
\end{definition}
Therefore, choosing $\alpha=\langle q^h\rangle$ where $q,h\in\mb{N}_0$, we have \[|\Lambda|_N(k;\langle q^h\rangle)=\int_{\mathrm{U}(N)}\left|\Lambda_U(1)\right|^{2k-h}\left|\Lambda_U^{(q)}(1)\right|^{h}\mathrm{d}U,\] and \[|\Lambda|_N(k;\emptyset)=\int_{\mathrm{U}(N)}\left|\Lambda_U(1)\right|^{2k}\mathrm{d}U,\] where $\emptyset$ denotes the empty set partition. The latter averages are the moments studied in the work of Keating and Snaith \cite{kn:keasna00a}, where they proved that for $k\in\mb{C}$ with $\mathfrak{R}(k)>-1/2$,
\begin{equation}\label{KS moments 1}
    |\Lambda|_N(k;\emptyset)= \frac{G(N+2k+1)G(N+1)G(k+1)^2}{G(N+k+1)^2 G(2k+1)},
\end{equation}
and
\begin{equation}
    \lim_{N\to\infty}  \frac{|\Lambda|_N(k;\emptyset)}{N^{k^2}}= \frac{G(k+1)^2}{G(2k+1)},
\end{equation}
where $G(z)$ is the Barnes $G$-function.

The leading order term of the joint moments is of particular interest. Thus, we define
\begin{align*}
 &\mathcal{M}(k;\alpha)\coloneq\lim_{N\to\infty}\frac{\mathcal{M}_N(k;\alpha)}{N^{k^2+|\alpha|}},\\
    &|\Lambda|(k;\alpha)\coloneq\lim_{N\to\infty}\frac{|\Lambda|_N(k;\alpha)}{N^{k^2+|\alpha|}},\\
    &|\mathcal{Z}|(k;\alpha)\coloneq\lim_{N\to\infty}\frac{|\mathcal{Z}|_N(k;\alpha)}{N^{k^2+|\alpha|}}.
\end{align*}
By the rotational invariance of the unitary group, the quantities $|\Lambda|_N(k;\alpha)$ and $|\mathcal{Z}|_N(k;\alpha)$ depend only on the modulus of the argument of the characteristic polynomial, so we may take it to be any fixed point on the unit circle.

A theme of our results is that they have two different combinatorial formulations. We recall that the \emph{Kostka numbers} $K_{\lambda\mu}$ is the number of semistandard Young tableaux of shape $\lambda$ and content $\mu$, and the \emph{hook number} $H_\lambda$ is the product of the hook lengths (see Subsection \ref{rep} for more details).

We now introduce our first main result, which establishes both exact and asymptotic expressions for $\mathcal{M}_N(k;\alpha)$.
\begin{maintheorem}\label{Theorem A}
         For a partition $\alpha=\langle1^{h_1}2^{h_2}\ldots l^{h_l}\rangle$ with $\sum_{i=1}^lh_i=h\leq 2k$,
     \begin{align*}
         \mathcal{M}_N(k;\alpha)&=\alpha!(-1)^{|\alpha|}|\Lambda|_N(k;\emptyset)\sum_{\substack{\\ \mu\vdash_{k}|\alpha|}}K_{\mu\alpha}\frac{ [-N]_\mu[k]_\mu}{H_\mu [2k]_\mu},
     \end{align*}
     where the sum is over partitions $\mu$ with at most $k$ parts and $|\alpha|=\sum_{i=1}^lih_i$.
\end{maintheorem}
\begin{maincorollary}\label{Cor A}
  For a partition $\alpha=\langle1^{h_1}2^{h_2}\ldots l^{h_l}\rangle$ with $\sum_{i=1}^lh_i=h\leq 2k$,
    \begin{align*}
        \mathcal{M}(k;\alpha)&=\alpha!|\Lambda|(k;\emptyset)\sum_{\substack{\\ \mu\vdash_{k}|\alpha|}}K_{\mu\alpha}\frac{ [k]_\mu}{H_\mu [2k]_\mu}.
    \end{align*}
\end{maincorollary}
Our second main result gives an alternative expression for the joint moments, via counting integer matrices.
If $A=(a_{i,j})$ is an $m\times n$ matrix with nonnegative integer values, its \emph{row-sum vector} $\operatorname{row}(A)\in\mb{Z}^m$ and \emph{column-sum vector} $\operatorname{col}(A)\in\mb{Z}^n$ are given by
\[\operatorname{row}(A)_i=\sum_{j=1}^na_{i,j}, \quad \quad \operatorname{col}(A)_j=\sum_{i=1}^ma_{i,j}.\]
Given two weak compositions $\beta=(\beta_1,\ldots,\beta_m)$ and $\tilde{\beta}=(\tilde{\beta}_1,\ldots,\tilde{\beta}_n)$, as in \cite{kn:diagam04}, we denote by\footnote{The notation $N_{\beta,\tilde{\beta}}$ should not be confused with the size of the matrix.} $N_{\beta,\tilde{\beta}}$ the number of $m\times n$ matrices $A$ with nonnegative integer values that satisfy $\operatorname{row}(A)=\beta$ and $\operatorname{col}(A)=\tilde{\beta}$. In the case that $m=n=k$ and $\beta_1=\cdots =\beta_k=\tilde{\beta_1}=\cdots =\tilde{\beta}_k$, $N_{\beta,\tilde{\beta}}$ counts the number of \emph{magic squares of order $k$}.
\begin{maintheorem}\label{Theorem B}
     For a partition $\alpha=\langle1^{h_1}2^{h_2}\ldots l^{h_l}\rangle$ with $\sum_{i=1}^lh_i=h\leq 2k$,
     \begin{align*}
         \mathcal{M}_N(k;\alpha)
         &=\alpha!(-1)^{|\alpha|}|\Lambda|_N(k;\emptyset)\sum_{\substack{\mathfrak{m}=(m_1,\ldots,m_k)\\ \sum_i^km_i=|\alpha|}}N_{\alpha,\mathfrak{m}}\frac{[-N]_\mathfrak{m}}{[2k]_\mathfrak{m}}\frac{\prod_{1\le i<j\le k}(m_i-m_j+j-i)}{\prod_{i=1}^k(k-i)!},
     \end{align*}
     where the sum is over weak compositions $\mathfrak{m}$ of length $k$ and size $|\alpha|$.
\end{maintheorem}
\begin{maincorollary}\label{Cor B}
  For a partition $\alpha=\langle1^{h_1}2^{h_2}\ldots l^{h_l}\rangle$ with $\sum_{i=1}^lh_i=h\leq 2k$,
    \begin{align*}
        \mathcal{M}(k;\alpha)=\alpha!\sum_{\substack{\mf{m}=(m_1,\ldots,m_k)\\ \sum_i^km_i=|\alpha|}}N_{\alpha,\mf{m}}\frac{\prod_{1\le i<j\le k}(m_i-m_j+j-i)}{\prod_{i=1}^k(m_i+2k-i)!}
    \end{align*}
\end{maincorollary}
\begin{remark}
    Theorem \ref{Theorem B} is a consequence of a duality between Kostka numbers and counting integer matrices (magic squares), which we will prove in Section \ref{section 3}. The precise statement of this duality is given in Lemma \ref{Lemma KostkaMagic}. It states that we may rewrite our sum in Theorem \ref{Theorem A} as a sum over weak compositions instead of partitions, whilst replacing the Kostka numbers with magic squares.
    
   The expression in Theorem \ref{Theorem B} thus follows from Lemma \ref{Lemma KostkaMagic} and the identity
    \[\frac{[k]_\lambda}{H_\lambda}=\frac{\prod_{i<j}(\lambda_i-\lambda_j-i+j)}{\prod_{i}(k-i)!},\]
    which shall be explained in Subsection \ref{rep}.
\end{remark}

Our methods dictate that we restrict our investigation of $|\Lambda|_N(k;\alpha)$ to even moments, by which we mean all exponents of the characteristic polynomials and their derivatives are even integers.  Therefore, in order to index these averages, $2\!\cdot\!\alpha$ henceforth denotes the partition \[2\!\cdot\!\alpha\equiv \alpha\cup\alpha=(\alpha_1,\alpha_1,\alpha_2,\alpha_2,\ldots,\alpha_{\ell(\alpha)},\alpha_{\ell(\alpha)}).\]

The following theorems prove that, in order to compute ${|\Lambda|_N(k;2\!\cdot\!\alpha)}$ and $ {|\mathcal{Z}|_N(k;2\!\cdot\!\alpha)}$, it suffices to compute $\mathcal{M}_N(k;\beta)$ where $\beta\le 2\!\cdot\!\alpha$.
\begin{theorem}\label{Theorem 1}
For a partition $\alpha$ of length $\le k$,
\begin{equation}
\begin{split}
    {|\Lambda|_N(k;2\!\cdot\!\alpha)}=\sum_{\substack{\beta\leq \alpha\\ \gamma\leq \beta}}(N)_{\alpha-\beta}(-1)^{|\beta|} \binom{\alpha}{\beta}L(\beta,\gamma)\mathcal{M}_N(k;\alpha\cup\gamma),
    \end{split}
\end{equation}    and 
\begin{equation}\label{1.1 asymp}
     {|\Lambda|(k;2\!\cdot\!\alpha)}=\sum_{\substack{\beta\leq \alpha\\ }}(-1)^{|\beta|}\binom{\alpha}{\beta}\mathcal{M}(k;\alpha\cup\beta),
\end{equation}
where $\binom{\alpha}{\beta}=\frac{\alpha!}{\beta!(\alpha-\beta)!}$ and $L(\beta,\gamma)$ are the \emph{(multi-index) Lah numbers} given by
\[L(\beta,\gamma)=\prod_iL(\beta_i,\gamma_i)=\prod_i\binom{\beta_i-1}{\gamma_i-1}\frac{\beta_i!}{\gamma_i!}.\]
\end{theorem}
\begin{theorem}\label{Theorem 2}
    For a partition $\alpha$ of length $\le k$,
\begin{equation}
\begin{split}
    {|\mathcal{Z}|_N(k;2\!\cdot\!\alpha)}=(-1)^{|\alpha|}\sum_{\substack{\beta\leq 2\cdot\alpha\\ \gamma\leq \beta}}\left(\frac{-N}{2}\right)^{2|\alpha|-|\beta|}\binom{2\!\cdot\!\alpha}{\beta}{\beta\brace \gamma}\mathcal{M}_N(k;\gamma),
    \end{split}
\end{equation}    and 
\begin{equation}\label{asymp 1.2}
     {|\mathcal{Z}|(k;2\!\cdot\!\alpha)}=(-1)^{|\alpha|}\sum_{\substack{\beta\leq 2\cdot\alpha\\ }}\left(\frac{-1}{2}\right)^{2|\alpha|-|\beta|}\binom{2\!\cdot\!\alpha}{\beta}\mathcal{M}(k;\beta),
\end{equation}
where $\beta\brace \gamma$ are the \emph{(multi-index) Stirling numbers of the second kind}.
\end{theorem}
Therefore, we can express $|\Lambda|_N(k;2\!\cdot\!\alpha)$ and $|\mathcal{Z}|_N(k;2\!\cdot\!\alpha)$ (as well as their corresponding asymptotics), in terms of Kostka numbers or magic squares. For instance, by combining \eqref{1.1 asymp} and Corollary \ref{Cor B}, we obtain
\begin{align}\label{conj1 eq}
   |\Lambda|(k;2\!\cdot\!\alpha)=\alpha!&\sum_{\substack{\beta\leq \alpha\\ }}\left(-1\right)^{|\beta|}\frac{1}{(\alpha-\beta)!}\\
   &\nonumber\times\sum_{\substack{\mf{m}=(m_1,\ldots,m_k)\\ \sum_i^km_i=|\alpha|+|\beta|}}N_{\alpha\cup\beta,\mf{m}}\frac{\prod_{1\le i<j\le k}(m_i-m_j+j-i)}{\prod_{i=1}^k(m_i+2k-i)!},
\end{align}
whereas, combining \eqref{asymp 1.2} and Corollary \ref{Cor B} gives
\begin{align}\label{conj2 eq}
    {|\mathcal{Z}|(k;2\!\cdot\!\alpha)}= (-1)^{|\alpha|}(\alpha!)^2&\sum_{\substack{\beta\leq 2\cdot\alpha\\ }}\left(\frac{-1}{2}\right)^{2|\alpha|-|\beta|}\frac{1}{(2\!\cdot\!\alpha-\beta)!}\\
    &\nonumber\times\sum_{\substack{\mf{m}=(m_1,\ldots,m_k)\\ \sum_i^km_i=|\beta|}}N_{\beta,\mf{m}}\frac{\prod_{1\le i<j\le k}(m_i-m_j+j-i)}{\prod_{i=1}^k(m_i+2k-i)!}.
\end{align}
Following the random matrix theory framework developed in \cite{kn:keasna00a,kn:cfkrs,kn:ghk}, we formulate the following conjecture concerning moments of higher-order derivatives of the Riemann zeta function $\zeta(s)$ on the critical line $\mathfrak{R}(s)=1/2$ , together with the corresponding moments for Hardy’s \(Z\)-function, defined by
\begin{equation}
    Z(t)=\frac{\pi^{-\mf{i}t/2}\Gamma\left(1/4+\mf{i}t/2\right)}{\left|\Gamma\left(1/4+\mf{i}t/2\right)\right|}\zeta(1/2+\mf{i}t).
\end{equation}

\begin{conjecture}\label{zeta conjecture}
For $k,h,l\in\mb{N}_0$ with $h\leq k$ and $(h_i)_{i=1}^l$ a sequence of non-negative integers which satisfy $\sum_{i=1}^l h_i =h$, set $\alpha$ to be the partition given by $\alpha=\langle1^{h_1}2^{h_2}\ldots l^{h_l}\rangle$. Then,
    \begin{align*}
        \frac{1}{T}\int_{0}^T\left|\zeta(1/2+\mf{i}t
    )\right|^{2k-2h}\prod_{i=1}^{l}\left|{\zeta^{(i)}(1/2+\mf{i}t)}\right|^{2h_i}\mathrm{d}t\sim a_k \left|\Lambda\right|\!(k;2\!\cdot\!\alpha)\left(\log T\right)^{k^2+2|\alpha|},
    \end{align*}
    and 
    \begin{align*}
         \frac{1}{T}\int_{0}^T\left|Z(t
    )\right|^{2k-2h}\prod_{i=1}^{l}\left|{Z^{(i)}(t)}\right|^{2h_i}\mathrm{d}t\sim a_k \left|\mathcal{Z}\right|\!(k;2\!\cdot\!\alpha)\left(\log T\right)^{k^2+2|\alpha|},
    \end{align*}
    where $a_k$ is the arithmetic factor
    \[a_k=
\prod_p
\left(1-\frac{1}{p}\right)^{k^2}
\sum_{m=0}^{\infty}
\left(
\frac{\Gamma(m+k)}{m!\,\Gamma(k)}
\right)^2
p^{-m},\]
    and
    $\left|\Lambda\right|\!(k;2\!\cdot\!\alpha)$ and $\left|\mathcal{Z}\right|\!(k;2\!\cdot\!\alpha)$ are given by \eqref{conj1 eq} and \eqref{conj2 eq}, respectively.
\end{conjecture}
\begin{remark}
    In \cite{kn:keawei1}, Keating and Wei computed $|\Lambda|(k;2\!\cdot\!\alpha)$ and $|\mathcal{Z}|(k;2\!\cdot\!\alpha)$ when $\alpha$ is a partition of $k$ into at most two distinct parts. That is, $\alpha=\langle n_1^{m},n_2^{k-m}\rangle$ where $n_1,n_2\in\mb{N}_0$ and $0\le m\le k$. This led them to conjecture (see Conjecture 1.7 of \cite{kn:keawei1}) the leading order term of
    \[   \frac{1}{T}\int_{0}^T\left|\zeta^{(n_1)}(1/2+\mf{i}t
    )\right|^{2m}\left|{\zeta^{(n_2)}(1/2+\mf{i}t)}\right|^{2k-2m}\mathrm{d}t,\]
    as $T\to\infty$. They made an analogous conjecture for the corresponding moments of Hardy’s \(Z\)-function.
    
    In Subsection \ref{section compare}, we use Theorem \ref{Theorem B} to extend their expressions for $|\Lambda|(k;2\!\cdot\!\alpha)$ and $|\mathcal{Z}|(k;2\!\cdot\!\alpha)$ to a general partition $\alpha$ of length $\le k$. Therefore, Conjecture \ref{zeta conjecture} can be seen as a generalisation of Conjecture 1.7 of \cite{kn:keawei1}.
    \end{remark}
\subsection{Broader motivation}

One of the motivations for the joint moment problem with $|z|<1$ is that a sufficient understanding of these averages, especially in the case the exponents in \eqref{eq:1} are non-integer, would reveal the underlying distribution of the zeros of the derivatives of a characteristic polynomial drawn randomly from the unitary group \cite{farmer2025bimodaldistributionderivativeunitary}.

The distribution in the \emph{global regime}, where $z$ lies a macroscopic distance away from the unit circle, was discovered by Mezzadri \cite{kn:mezzadri03}. In \cite{kn:simmwei26}, Simm and Wei recovered Mezzadri's result by using the method of joint moments. The \emph{microscopic regime} is where $z$ lies on the order of the mean separation of eigenvalues, $1/N$, away from the unit circle. This distribution remains unknown, but has potential applications to number theory \cite{kn:consna06}, as well as random matrix theory. 

Computing the joint moments for the microscopic regime, when the exponents are non-integers, remains an open problem. However, in recent work, Simm and Wei \cite{kn:simmwei26}, have related the moments of the derivative of the characteristic polynomial in the microscopic regime (with even integer exponents) to the finite-temperature Bessel kernel. Furthermore, they provided a combinatorial interpretation of these moments, which we state in Theorem \ref{SW theorem}.

Whilst our approach currently restricts us to moments on the unit circle with even exponents, we hope to convey that these moments have a rich combinatorial structure, with connections to the $\tau$-functions of Painlev\'e, which are yet to be fully understood.
\subsection{History of the problem} 
The joint moment problem has been intensively studied for over two decades. Due to the several different methods one can use in computing the averages given in Definition \ref{def 1}, there is now an array of different answers to the same problem. It is an interesting and important exercise to explain the connection between these results. 

In this paper, we unify several pre-existing results for even moments when the characteristic polynomial is evaluated on the unit circle. This includes an irregular conformal block structure as predicted in \cite{kn:basor_et_al18} (see Subsection \ref{irreg conformal block section} for further details). We interpret our restricted investigation as a representation theoretic problem, as this shall explain the equivalence of Theorem \ref{Theorem A} and Theorem \ref{Theorem B}. In Subection \ref{unifying section}, we provide a detailed description of how these theorems logically imply many other results. Unfortunately, a unified approach to a more generalised joint moment problem (for instance, where one considers non-integer exponents) remains out of reach. This is principally due to the fact that there is no known expression for averages of characteristic polynomials with fractional powers in terms of Schur polynomials.

We shall now survey the literature, first highlighting the results that are more relevant to us. These will have an underlying combinatorial structure. Then, we will give a short review of what is known for the joint moment problem outside of our restricted regime, and which cannot currently be explained using our methods.
\begin{definition}
    Let $k$ and $h$ be non-negative integers with $h\leq 2k$. We define
    \begin{align*}
      |\mathcal{Z}|_N(k,h)&\coloneq|\mathcal{Z}|_N(k;\langle1^h\rangle),\\
      |\mathcal{Z}|(k,h)&\coloneq|\mathcal{Z}|(k;\langle1^h\rangle),
    \end{align*}
and define $|\Lambda|_N(k,h)$ and $\mathcal{M}_N(k,h)$ analogously.
\end{definition}
The joint moment problem was first considered in Hughes' PhD thesis \cite{kn:hug01}. Conrey, Rubinstein and Snaith \cite{kn:crs06} were the first to compute an explicit formula for the leading order term of \[\int_{\mathrm{U}(N)}\left|\mathcal{Z}_U'(0)\right|^{2k}dU,\]
as $N\to\infty$. Furthermore, they related the first order joint moments to a determinant of Bessel functions.
\begin{theorem}[Theorem 2 of \cite{kn:crs06}]\label{crs 1} For $k\in\mathbb{N}_0$,
    \begin{equation}\label{eq crs 1}
    |\mathcal{Z}|(k,2k)= (-1)^{k(k+1)/2} \left( \frac{\mathrm{d}}{\mathrm{d}x} \right)^{2k} \left( e^{-\frac{x}{2}} x^{-k^2/2} \det_{k \times k} (I_{i+j-1}(2\sqrt{x})) \right) \bigg|_{x=0},
\end{equation}
where $I_n(x)$ is the modified Bessel function of the first kind.
\end{theorem}
$I_n(x) $ is given by the series:
\[I_n(x)=\left(\frac{x}{2}\right)^n\sum_{j=0}^\infty\frac{x^{2j}}{2^{2j}(n+j)!j!}.\]
Using this expansion, they also provided a combinatorial expression (Theorem 3 of \cite{kn:crs06}),
\begin{equation}\label{eq crs 2}
    \begin{split}
|\mathcal{Z}|(k,2k)=&(-1)^k\sum_{\substack{\sum_{i=0}^{k}m_i=2k}}\binom{2k}{m_0,\ldots,m_{k}}\left(\frac{-1}{2}\right)^{m_0}\\
&\times \frac{\prod_{1\le i<j\le k}(m_i-m_j+j-i)}{\prod_{i=1}^k(m_i+2k-i)!}.
    \end{split}
\end{equation}
The authors' approach was to take the first order derivative of averages of characteristic polynomials with shifts. The expression for averages of characteristic polynomials that they use is in terms of a multidimensional contour integral.

In \cite{kn:dehaye10p}, Dehaye used a different method to prove an exact expression for $|\mathcal{Z}|_N(k,2h)$, where $h\le k$. This viewpoint is taken throughout this paper.
\begin{theorem}[Dehaye \cite{kn:dehaye10p}]\label{dehaye theorem} For $k,h\in\mathbb{N}_0$ where $h\le k$,
   \begin{align*}
       |\mathcal{Z}|_N(k,2h)=(-1)^h|\Lambda|_N(k,0)\sum_{r=0}^{2h}&\binom{2h}{r}\left(\frac{-N}{2}\right)^{2h-r}
       (-1)^rr!\sum_{\substack{\\ \mu\vdash_{k}r}}\frac{ [-N]_\mu[k]_\mu}{H^2_\mu [2k]_\mu}.
   \end{align*}
\end{theorem}
In particular, he proved that the $2k$th moments of $|\Lambda_U(1)|$, i.e. $|\Lambda|_N(2k,0)$, are a polynomial factor of these joint moments.

As a corollary of Theorem \ref{dehaye theorem}, one easily obtains the corresponding asymptotic:
\begin{equation}\label{dehaye asymptotic}
    |\mathcal{Z}|(k,2h)=(-1)^h|\Lambda|(k,0)\sum_{r=0}^{2h}\binom{2h}{r}\left(\frac{-1}{2}\right)^{2h-r}r!\sum_{\substack{\\ \mu\vdash_{k}r}}\frac{[k]_\mu}{H^2_\mu [2k]_\mu}.
\end{equation}
Understanding the relation between Dehaye's result \eqref{dehaye asymptotic} and \eqref{eq crs 2} was the starting point for this paper. Observe that Theorem \ref{dehaye theorem} follows from Theorem \ref{Theorem A} by using the identity (see \cite{Macdonald})
\[K_{\mu\langle 1 ^r\rangle}=\frac{r!}{H_\mu}, \]
which can be derived directly from Definition \ref{def kostka} and the hook length formula.

We will now show that Theorem \ref{crs 1} (or more precisely \eqref{eq crs 2}) follows from Theorem B, hence motivating the need for both interpretations. It is useful to rewrite (\ref{eq crs 2}) as
\begin{equation}
\begin{split}
    |\mathcal{Z}|(k,2k)=&(-1)^k\sum_{r=0}^{2k}\binom{2k}{r}\left(\frac{-1}{2}\right)^{2k-r}\\
    &\times\sum_{\substack{\sum_{i=1}^{k}m_i=r}}\binom{r}{m_1,\ldots,m_{k}}
 \frac{\prod_{1\le i<j\le k}(m_i-m_j+j-i)}{\prod_{i=1}^k(m_i+2k-i)!}.
\end{split}
\end{equation}
This reformulation follows from making the substitution $m_0=2k-r$. It is a simple check from the definition that
\begin{equation}
    N_{\mf{m}\langle 1 ^r\rangle}=\binom{r}{m_1,\ldots,m_k},
\end{equation}
where $\mf{m}=(m_1,\ldots,m_k)$.

Therefore, the equivalence of (\ref{dehaye asymptotic}) and (\ref{eq crs 2}) follows from the relations
\begin{align*}
    {\mathcal{M}(
   k,r)}{}&=|\Lambda|(k,0)r!\sum_{\substack{\\ \mu\vdash_{k}r}}\frac{ [k]_\mu}{H^2_\mu [2k]_\mu}\\
   &=\sum_{\substack{\sum_{i=1}^{k}m_i=r}}\binom{r}{m_1,\ldots,m_{k}}
 \frac{\prod_{1\le i<j\le k}(m_i-m_j+j-i)}{\prod_{i=1}^k(m_i+2k-i)!},
\end{align*}
where the first equality follows from Corollary \ref{Cor A} and the second from Corollary \ref{Cor B}.\\

Recently, using a similar method to \cite{kn:crs06}, Keating and Wei \cite{kn:keawei1, kn:keawei2} have calculated formulas for $|\Lambda|(k;2\!\cdot\!\alpha)$ and $|\mathcal{Z}|(k;2\!\cdot\!\alpha)$ where $\alpha$ is a partition of $k$ into at most two distinct parts. 
 
  In \cite{kn:keawei1}, the authors observe that for higher order joint moments, two different expressions are needed. The first expression they derive is given in terms of an extensive sum of determinants of Bessel functions. The importance of these expressions is discussed in Subsection \ref{section compare}. Whereas, the second formula they derive is more computationally efficient (Theorem 3.7 of \cite{kn:keawei1}). These shall be reproven in Subsection \ref{section compare} and extended to an arbitrary partition of length $\le k$.

  Recently, \cite{kn:acrs} proved an exact formula for $|\Lambda|_N(k,2k)$, involving determinants of binomial coefficients (see our Theorem \ref{ACRS theorem1}).
  \subsubsection{Other regimes}\label{other regimes}
We will give a brief review of the progress on the joint moment problem for the regimes which lie outside the scope of this paper. In particular, these consider averages of the form
\[\int_{\mathrm{U}(N)}\left|\mathcal{Z}_U(\theta)\right|^{2k-\sum_{i=1}^l ih_i}\prod_{i=1}^l\left|\mathcal{Z}_U^{(i)}(\theta)\right|^{h_i}\mathrm{d}U,\]
where $k\notin\mb{Z}$ or $h_i/2\notin \mb{Z}$ or $\theta\in\mb{C}\setminus\mb{R}$. Alternatively, averages of $\Lambda_U(z)$ are considered.

For $\theta\in\mb{R}$, Winn \cite{kn:winn12} extended Dehaye's result, our Theorem \ref{dehaye theorem}, by computing $|\mathcal{Z}|_N(k,h)$ for $k\in\mb{Z}$ and $h$ an odd positive integer. For any real $2k>-1$ and $0<h<2k+1$ Assiotis,
Keating, and Warren \cite{kn:akw22} proved that the limit
\[\lim_{N\to\infty}\frac{1}{N^{k^2+h}}\int_{\mathrm{U}(N)}\left|\mathcal{Z}_U(0)\right|^{2k-h}\left|\mathcal{Z}_U'(0)\right|^{h}\mathrm{d}U,\] exists, and they gave a probabilistic interpretation of the limit in
terms of the expectation of a real-valued random variable with respect to the Hua–Pickrell measure. In \cite{kn:agkw}, Assiotis, Gunes, Keating and Wei extended this to higher order joint moments, by proving convergence of
\[\frac{1}{N^{k^2+\sum_{i=1}^lih_i}}\int_{\mathrm{U}(N)}\left|\mathcal{Z}_U(0)\right|^{2k-h}\prod_{i=1}^l\left|\mathcal{Z}_U^{(i)}(0)\right|^{h_i}\mathrm{d}U, \quad k,h_i\in\mb{R}_{>0}, \ h=\sum_{i=1}^lh_i\le 2k,\]
to the expectation of a certain random variable. 

Recently, for $2k>-1$ and $h\in\mb{C}$ with $0\le \mathfrak{R}(h)<2k+1$, Bothner and Wei \cite{bothnerwei} have related the limit
\[\lim_{N\to\infty}\frac{1}{N^{k^2+h}}\int_{\mathrm{U}(N)}\left|\mathcal{Z}_U(0)\right|^{2k-h}\left|\mathcal{Z}_U'(0)\right|^{h}\mathrm{d}U,\]
to a solution of a Painlev\'e III\'{}, with explicit form given in terms of a unique solution to a Riemann-Hilbert problem.

For a fixed $z$ with $|z|<1$ and any $k\in\mb{C}$ and $\mathfrak{R}(k)>-1$, Simm and Wei \cite{kn:simmwei26} have given an explicit expression for
\[\lim_{N\to\infty}\int_{\mathrm{U}(N)}\left|\Lambda_U'(z)\right|^{2k}\mathrm{d}U,\]
in terms of a confluent hypergeometric function of the first kind. Furthermore, they also proved the following asymptotic for the \emph{microscopic} regime. This is of particular interest due to its connections to the radial distribution of zeros of the derivative of the characteristic polynomial.
\begin{theorem}[Theorem 1.5 of \cite{kn:simmwei26}]\label{SW theorem}
  Let $|z_N|^2 = 1 - \frac{c}{N}$ for $c \in \mathbb{R}_{\ge0}$ fixed. In particular, the case $|z_N| = 1$ corresponding to $c = 0$ is allowed. Then, for any positive integer $k$, as $N \to \infty$, we have
\begin{equation}\label{simmwei}
\begin{split}\int_{\mathrm{U}(N)}\left|\Lambda_U'(z_N)\right|^{2k}\mathrm{d}U\sim & \ N^{k^2+2k} \frac{1}{\left(\prod_{i=1}^k(k-i)!\right)^2}\\
    &\times\sum_{\lambda,\mu\vdash k}\frac{f^\lambda f^\mu}{[k]_\lambda [k]_\mu} \det_{k\times k}\left(\int_0^1 x^{\lambda_i + \mu_j + 2k - i - j} e^{-cx} \, \mathrm{d}x\right),
    \end{split}
\end{equation}
where $f^\lambda$ denotes the number of standard Young tableaux of shape $\lambda$.
\end{theorem}

Theorem \ref{SW theorem} has a similar structure to Theorem \ref{dehaye theorem} after recalling that $f^\lambda=K_{\lambda\langle1^k\rangle}$. In order to relate this combinatorial expression to the material discussed in this paper, especially the irregular conformal block structure, it is an interesting open challenge to prove directly the equivalence of expressions for \eqref{eq crs 1} and \eqref{simmwei} when $c=0$.

Theorem \ref{SW theorem} has been generalised to higher order derivative joint moments in \cite{GMS26}, where they provided an expression also in terms of Kostka numbers.
\subsection{Unifying the literature on the problem}\label{unifying section}
This paper focuses on a restricted version of the joint moment problem \eqref{eq:1}, where all exponents are taken to be even integers and $z$ lies on the unit circle. However, many different formulas have been provided in this case.
  We will now remark on the connections between them (see Figure \ref{figure compare}), and they will be proven throughout the paper by analysing their underlying combinatorial structure. 
  \begin{remark}
   \begin{itemize}
       \item[{\textit{i})}] In Section \ref{section 3}, Theorem B will be proven from Theorem A as an application of the Jacobi-Trudi formula.
       \item[{\textit{ii})}] All the results in blue boxes in Figure \ref{figure compare} represent existing results and will be proved as a corollary of our main results except for Theorem 3.7 of \cite{kn:keawei1}, which expresses the joint moments as a sum of determinants of modified Bessel functions of the first kind. In Subsection \ref{section compare}, we will prove this result for some cases and with some more work, it should be possible to prove the general case.
      
   \end{itemize}   
  \end{remark}

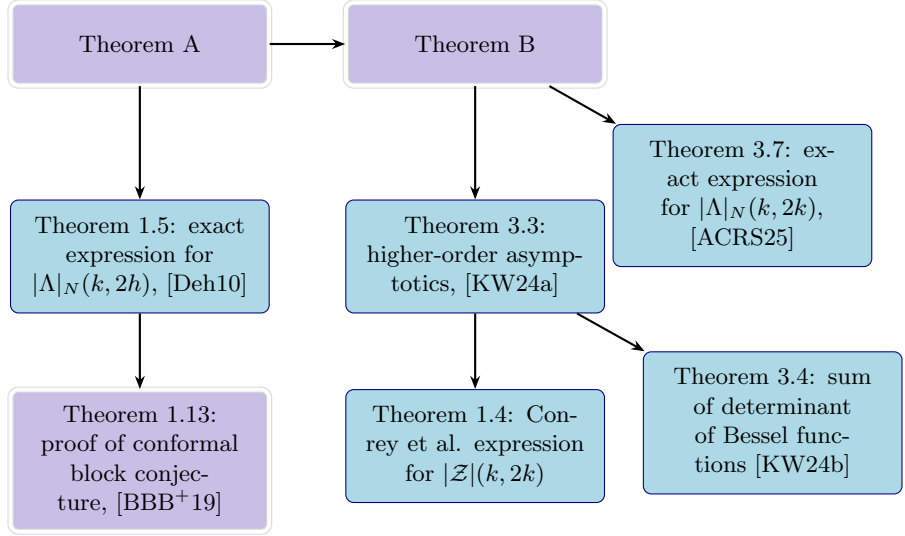
\begin{figure}[h]
\centering
\begin{tikzpicture}[
  % Base box style
  box/.style = {
    draw,
    rectangle,
    rounded corners=3pt,
    align=center,
    text width=3.0cm,
    minimum height=1.1cm,
    font=\small,
    inner sep=6pt
  },
  % Blue-filled boxes (most nodes)
  bluebox/.style = {
    box,
    fill=boxblue,
    draw=blue!50!black
  },
  % Gray-filled box (the central sieve node, double border in original)
  graybox/.style = {
    box,
    fill=boxgray,
    draw=gray!30,
    double,           % double border like in the figure
    double distance=1.5pt
  },
  arr/.style = {
    -{Stealth[length=5pt]},
    thick
  }
]

% ── Nodes ───────────────────────────────────────────────────────────────
\node[graybox] (kostka)
  {Theorem \ref{Theorem A}};

\node[graybox, right=1.0cm of kostka] (magic)
  {Theorem \ref{Theorem B}};

\node[bluebox, below =1.5cm of kostka] (dehaye)
  {Theorem \ref{dehaye theorem}: exact expression for $|\Lambda|_N(k,2h),$ \cite{kn:dehaye10p}};
  
  \node[bluebox, below =1.5cm of magic] (Keating-Wei)
  {Theorem \ref{Keating-Wei generalisation}: higher-order asymptotics, \cite{kn:keawei1}};

\node[bluebox, below right=0.5cm and 0.1cm of magic] (ACRS)
  {Theorem \ref{ACRS theorem1}: exact expression for $|\Lambda|_N(k,2k)$, \cite{kn:acrs}};

\node[graybox, below =1cm of dehaye] (conformal)
  {Theorem \ref{conj proof}: proof of conformal block conjecture, \cite{kn:basor_et_al18}};
  
  \node[bluebox, below right=0.5cm and 0.5cm of Keating-Wei] (Bessel)
  {Theorem \ref{KW Theorem 4}: sum of determinant of Bessel functions \cite{kn:keawei2}};

\node[bluebox, below=1cm of Keating-Wei] (CRS)
  {Theorem \ref{crs 1}: Conrey et al. expression for $|\mathcal{Z}|(k,2k)$};

% ── Arrows ──────────────────────────────────────────────────────────────
\draw[arr] (kostka)     -- (dehaye);
\draw[arr] (kostka)     -- (magic);
\draw[arr] (magic)     -- (ACRS);
\draw[arr] (magic)     -- (Keating-Wei);
\draw[arr] (dehaye)    -- (conformal);
\draw[arr] (Keating-Wei) -- (Bessel);
\draw[arr] (Keating-Wei) -- (CRS);

\end{tikzpicture}

\caption{ \textit{Structure of paper
  (arrows signify logical implications).
  The purple boxes are results original to this paper.}}
  \label{figure compare}
\end{figure}

\subsection{The representation theory of the unitary group}\label{rep}

It is often helpful to think of partitions graphically. The \emph{Young diagram} of a partition $\lambda$ is defined as the set of points $(i,j)\in \mb{Z}^2$ such that $1\leq i\leq \lambda_j$. It is convenient to replace each point in the lattice $(i,j)$ by a square, which we call a \emph{box}. For each partition $\lambda$, its transpose partition $\lambda'$ is the unique partition obtained by reflecting the Young diagram of $\lambda $ along the main diagonal. See Figure \ref{DiagramFig} for an example of the Young diagram of $\lambda= (4,4,2,1)$ and its transpose.
\begin{figure}
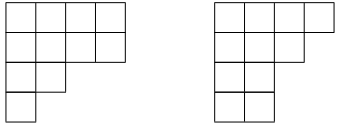

\centering
\yng(4,4,2,1) \quad \quad\quad\yng(4,3,2,2)
\caption{the Young diagrams of $\lambda= (4,4,2,1)$ (left) and its transpose $\lambda'=(4,3,2,2)$ (right).}
\label{DiagramFig}
\end{figure}

\begin{definition}
    For each partition $\lambda$ and for each box $(i,j)\in \lambda$, we define its \emph{arm length} $a\!\left(i,j\right)$ as the number of
boxes strictly to the right of $(i,j)$, and its \emph{leg length} $g\!\left(i,j\right)$ as the number of boxes strictly
below $(i,j)$. The \emph{hook length} of $(i,j)$ is then defined as 
\begin{equation}
    h_\lambda {(i,j)}=a\!\left(i,j\right)+g\!\left(i,j\right)+1,
\end{equation}
    and the \emph{hook number} is defined as
    \begin{equation}
        H_\lambda=\prod_{(i,j)\in \lambda}h_\lambda{(i,j)}.
    \end{equation}
\end{definition}
Equivalently,  the hook-length $h_\lambda{(i,j)}$ is given by
 \begin{equation}
        h_\lambda {(i,j)}=1+(\lambda_i-i)+(\lambda_j'-j).
    \end{equation}
    Hence, we see that
\[H_\lambda=H_{\lambda'}.\]
Furthermore, it is a useful observation that for a partition $\lambda$, its generalised Pochammer symbols satisfy the following:
\[[x]_\lambda =\prod_{i=1}^{\ell(\lambda)}\prod_{j=1}^{\lambda_i}(x+j-i)= \prod_{(i,j)\in\lambda}(x+j-i)=(-1)^{|\lambda|}[x]_{\lambda'} \ ,\]
where the last equality follows from reflecting the boxes in the diagonal of the Young diagram.
\begin{definition}[Young Tableaux]
    Let $\lambda=(\lambda_1,\ldots,\lambda_{\ell(\lambda)})$ be a partition of $n$, with $\ell(\lambda)\le n$. A \emph{Young tableau} of $\lambda$ is a filling of the boxes of the Young diagram of $\lambda$ with entries taken from a set of numbers.
    
    A Young tableau is \emph{standard}  if the entries in each row and each column are strictly increasing. It is \emph{semistandard} if the entries weakly increase along each row and strictly increase down each column.
\end{definition}

We shall make use of the following classical theorem (see \cite{Bumpbook}).
\begin{theorem}[Hook length formula]\label{hook length formula}
 Let $\lambda$ be a partition of $n$. The number of standard Young tableaux of shape $\lambda$, denoted by $f^\lambda$, is given by
 \begin{equation}
     f^\lambda=\frac{n!}{H_\lambda}.
 \end{equation}

\end{theorem}
We will now define the Kostka numbers.
\begin{definition}[Kostka numbers]\label{def kostka}
Let $\alpha$ be a weak composition and $\lambda$ be a partition. A Young tableaux of shape $\lambda$ has \emph{weight} $\alpha$ if its Young diagram is filled with $\alpha_1$ lots of $1$, $\alpha_2$ lots of $2$ and so on.

    The \emph{Kostka numbers} $K_{\lambda\alpha}$ are the number of semistandard Young tableaux of shape $\lambda$ and weight $\alpha$.
\end{definition}
We make two observations, which will be helpful for our analysis. Firstly, it is clear to see that $K_{\lambda\alpha}=0$ unless $|\alpha|=|\lambda|$. Secondly, if $\lambda$ is a partition of $n$ and $\alpha=\langle1^n\rangle$, then the entries along in each row and column must be strictly increasing. Therefore, in this case a semistandard Young tableau of shape $\lambda$ and weight $\alpha$ is a standard Young tableau and we have
\begin{equation}
    K_{\lambda\langle1^n\rangle}=f^\lambda=\frac{n!}{H_\lambda}.
\end{equation}
\begin{definition}\label{def schur}
For a partition $\lambda$ of length $\le n$, we define its Schur polynomial as the sum over monomials,
\[s_\lambda(x_1,\ldots,x_n)=\sum_Tx^T=\sum_{T}x_1^{t_1}\cdots x_n^{t_n},\]
    where the summation is over all semistandard Young tableaux $T$ of shape $\lambda$ filled with the numbers $1, 2, \ldots, n$, and  $t_i$ counts the occurrences of the number $i$ in $T$.
\end{definition}
It is well known that Schur polynomials may be written as the ratio of two $n\times n$ determinants,
\begin{equation}\label{schur determinant}
s_\lambda(x_1, \ldots, x_n) = \frac{ \det \left( x_i^{\lambda_j+n-j} \right) }{ \det \left( x_i^{n-j} \right)}.\end{equation}
This is the bialternant Cauchy–Jacobi formula for the Schur polynomials \cite{Macdonald}.
Our interest in these functions lies in their well known connection to the unitary group. Given a partition $\lambda$ of length $\leq n$, one defines a map from $\mathrm{GL}(n)$ to $\mb{C}$ in the following way. We define
\[\chi_\lambda ^{}(g)\coloneq s_\lambda (g)\coloneq s_{\lambda}(y_1,\ldots,y_n),\] where the $y_i$s are the eigenvalues of $g\in \mathrm{GL}(n)$. The following theorem is due to Weyl \cite{Weyl}, see also \cite{Bumpbook}.
\begin{theorem}
    If $\lambda$ is a partition of length $\leq n$, the function $\chi_{\lambda}$ is the
character of an irreducible analytic representation
$\pi_{\lambda}=\pi^{(n)}_{\lambda}$ of $\mathrm{GL}(n)$.
It is irreducible; in fact, its restriction to $\mathrm{U}(N)$ is irreducible.
\end{theorem}

The characters of the classical compact Lie groups are important tools for studying averages of characteristic polynomials due to \emph{Schur orthogonality}. Let $\langle \ , \ \rangle$ be the inner product with respect to the Haar measure on $\mathrm{U}(N)$. Then, if $\chi_\lambda$ and $\chi_{\mu}$ are any two irreducible characters, we have
\begin{equation}
   \langle \chi_\lambda,\chi_\mu\rangle_ =\begin{cases}
        1 \quad \text{if }\lambda=\mu,\\
        0 \quad \text{otherwise}.
    \end{cases}
\end{equation}

The Weyl dimension formula \cite{Weyl} gives
\begin{equation}\label{Weyl dimension}
    s_\lambda(\{1\}^n)=\frac{\prod_{i<j}(\lambda_i-\lambda_j-i+j)}{\prod_{i}(n-i)!},
\end{equation}
where $\{1\}^n$ denotes $n$ copies of $1$. The hook length formula for the number of semistandard Young tableaux gives an alternative formulation,
\begin{equation}\label{Hookeq}
    s_\lambda(\{1\}^n)=\frac{[n]_\lambda}{H_\lambda}.
\end{equation}

\subsubsection{The ring of symmetric functions}

For each $n\geq 1$, the symmetric group $S_n$ acts on the ring $\mb{Z}[x_1,\ldots,x_n]$ by permuting the variables. A \emph{symmetric function} is a polynomial, which is invariant under this action, and they form a subring.

    Let $\lambda$ be a partition of length $\leq n$. The \emph{monomial symmetric polynomial} is defined as 
    \[m_\lambda (x_1,\ldots,x_n)=\sum_{\beta \sim \lambda}x_1^{\beta_1}\cdots x_n^{\beta_n},\]
    where $\beta \sim \lambda$ if $\beta$ is a distinct permutation of the parts of $\lambda$.

    The $k$th \emph{complete homogeneous symmetric polynomial}
in $n $ variables, for a given $k\in\mb{N}$, is defined as
\begin{align*}
  h_k(x_1,\ldots,x_n)=\sum_{1\leq i_1\le i_2\le \cdots\le i_k\leq n}x_{i_1}x_{i_2}\cdots x_{i_k}.
\end{align*}
Now given a partition $\lambda=(\lambda_1,\ldots,\lambda_n)$, we define
\[h_\lambda(x_1,\ldots,x_n)=\prod_{i=1}^nh_{\lambda_i}(x_1,\ldots,x_n).\]

Schur polynomials, monomials and $h_\lambda$ are all bases of the ring of symmetric polynomials. Therefore, we may transition between $s_\lambda$ or $h_\lambda$ into the monomials. The transition matrices are given by Kostka numbers and $N_{\lambda \mu}$ respectively.
\begin{lemma}\label{Kostkalemma}
    Given a partition $\lambda$ of length $\leq n$, we have \begin{align*}&s_\lambda(x_1,\ldots,x_n)=\sum_{\mu\vdash|\lambda|}K_{\lambda \mu}m_\mu(x_1,\ldots,x_n),\\
    &h_\lambda(x_1,\ldots,x_n)=\sum_{\mu\vdash|\lambda|}N_{\lambda \mu}m_\mu(x_1,\ldots,x_n),\end{align*}
    where the sum is over partitions $\mu.$
\end{lemma}
We also need the well-known Jacobi-Trudi formula.
\begin{theorem}[Jacobi-Trudi identity]\label{JacobiTrudi}
    Let $\lambda$ be a partition of length $\leq n$. Then
    \[s_\lambda(x_1,\ldots,x_n)=\det_{\ell(\lambda)\times \ell(\lambda)}\left(h_{\lambda_i-i+j}(x_1,\ldots ,x_n)\right).\]
\end{theorem}
\subsection{The connection to irregular conformal blocks}\label{irreg conformal block section}
This paper gives a comprehensive account of the problem in our restricted regime. However, many of the results that go beyond this restriction use the theory of integrable systems. More precisely, they relate the problem to the theory of Painlev\'e transcendents. 

In \cite{kn:basor_et_al18}, Basor et al. conjectured a combinatorial expression for $|\mathcal{Z}|(k,2h)$ when $k\in\mb{N}_0$, based on its connection to Painlev\'e III' (see Conjecture 2 of \cite{kn:basor_et_al18}, or our Conjecture \ref{conj2} below). Their conjecture drew from a deep relationship between Painlev\'e III' and irregular conformal blocks, which for our purposes are explicit combinatorial series expansions (see \eqref{irreg block}).

 In this subsection, we will give a brief overview of the relation between joint moments and Painlev\'e III', and the conjectured connection between the latter and irregular conformal blocks. Moreover, we will prove that the combinatorial expression conjectured by Basor et al. is a logical implication of a result of Dehaye (our Theorem \ref{dehaye theorem}). In forthcoming work \cite{snape26}, we will prove a full series expansion for $\tau_k(t)$ for $k\in\mb{N}_0$, in terms of irregular conformal blocks. Furthermore, we will prove an analogous expansion for the finite $N$ joint moments, which will prove a modification of Conjecture 1 of \cite{kn:basor_et_al18}.\\

In \cite{kn:forwit06}, Forrester and Witte provided an expression for $|\Lambda|(k,2h)$ and $|\mathcal{Z}|(k,2h)$ in terms of solutions of $\sigma$-Painlev\'e III' equation, when $h=k\in\mb{N}_0$. This result has been generalised by Basor et al. \cite{kn:basor_et_al18} to $h\le k$, where $h$ and $k$ are non-negative integers, and then further extended by Assiotis et al. \cite{kn:abgs} to $h\in\mb{N}_0$ and any $k>h-1/2$. In particular, they proved the following.
\begin{theorem}[\cite{kn:abgs}]\label{assiotis theorem}
    For $h\in\mb{N}_0$ and for any $k>h-1/2$,
    \begin{equation}
        |\mathcal{Z}|(k,2h)=(-1)^h|\mathcal{Z}|(k,0)\frac{\mathrm{d}^{2h}}{\mathrm{d}t^{2h}}\tau_{k}(t)\big|_{t=0},
    \end{equation}
    where \begin{equation}\label{tau1 eq}\tau_k(t):=\exp\left(\int_{0}^t\frac{\sigma_k(u)}{u}\mathrm{d}u\right),\end{equation}
    and $\sigma_k$ satisfies the $\sigma$-Painlev\'e III' equation
     \begin{equation}\label{1.20}
         \left(t\frac{\mathrm{d}^2\sigma_k}{\mathrm{d}t^2}\right)^2=\left(\left(4\frac{\mathrm{d}\sigma_k}{\mathrm{d}t}\right)^2-1\right)\left(\sigma_k+k^2-t\frac{\mathrm{d}\sigma_k}{\mathrm{d}t}\right)+k^2.
     \end{equation}
\end{theorem}
Therefore, Theorem \ref{assiotis theorem} proves that for $h\in\mb{N}_0$, $|\mathcal{Z}|(k,2h)$ coincides with the coefficients in the series expansion of $\tau_k(t)$. Furthermore, $\tau_k(t)$, as defined in \eqref{tau1 eq}, is a distinguished function as it is precisely a $\tau$-function of Painlev\'e III', as defined in \cite{Gamayun_2013}.

In a series of seminal papers (see for example \cite{Gamayun_2012, Gamayun_2013}), it has been conjectured, and in many cases proven \cite{Iorgov_2014,Lisovyy_2018}, that the $\tau$-functions of Painlev\'e are solved by combinatorics.

The conjectured expansion for the $\tau$-function corresponding to the form of the $\sigma$-Painlev\'e III' equation given in \eqref{1.20} is (see Conjecture 5 of \cite{Gamayun_2013})
\begin{equation}\label{tau3 expansion}
\tau_k(t) = \mathcal{N}_0 e^{-\frac{t}{2}} \sum_{n\in\mb{Z}}^{\infty}s_{\mathrm{III}'}^n C_{\text{III}'}(-k, \sigma+n)\, t^{2nk+n^2} B_{\mathrm{III}'}\left(-k,\sigma+n;t\right),
\end{equation}
where $\mathcal{N}_0$ is a constant, $\sigma$ and $s_{\mathrm{III}}$ are complex parameters, which play the role of integration constants,
\begin{equation}\label{structure iii}
\begin{split}
&C_{\text{III'}}(-k, \sigma)\coloneq  \prod_{\epsilon=\pm1}\frac{G(1-k+\epsilon\sigma)G(1+\epsilon\sigma)^2}{  G(1+2\epsilon\sigma)^2},
\end{split}
\end{equation}
are called the \emph{structure constants},
and \(B_{\mathrm{III'}}\left(-k,\sigma;t\right)\) are called \emph{irregular conformal blocks}. They are given by the following series expansion
\begin{equation}\label{irreg block}
    B_{\mathrm{III'}}\left(-k,\sigma;t\right)\coloneq\sum_{\lambda,\mu \in  \mathrm{P}} B^{\text{III'}}_{\lambda,\mu}(-k, \sigma+n)\, t^{|\lambda|+|\mu|},
\end{equation}
where $\mathrm{P}$ denotes the set of partitions and for partitions $\lambda$ and $\mu$, we define the combinatorial quantity
\begin{equation}\label{eq conformal block3}
\begin{split}
B^{\mathrm{III}'}_{\lambda,\mu}(-k, \sigma) \coloneq  &\prod_{(i,j)\in\lambda} \frac{(-k+\sigma+i-j)(\sigma+i-j)}{h_\lambda^2(i,j)\left(\lambda_j' + \mu_i - i - j + 1 + 2\sigma\right)^2}\\
&\times \prod_{(i,j)\in\mu} \frac{(-k-\sigma+i-j)(-\sigma+i-j)}{h_\mu^2(i,j)\left(\lambda_i+\mu_j'-i-j+1-2\sigma\right)^2}.
\end{split}
\end{equation}
In \cite{kn:basor_et_al18}, the authors found that $\sigma=k$. However, when $2\sigma\in\mb{Z}$, the structure constants in \eqref{structure iii} are undefined due to the zeros of the Barnes $G$-function. To overcome this, they defined new structure constant by taking an appropriate limit of the Barnes $G$-function. We refer the reader to Section 6 of \cite{kn:basor_et_al18} for more details. This reasoning led Basor et al. to make the following conjecture.
\begin{conjecture}[Conjecture 2 of \cite{kn:basor_et_al18}]\label{conj2} Let $k,h\in\mb{N}_0$ with $h\le k$. Then,\footnote{There is a typo in (6-32) of \cite{kn:basor_et_al18} in that the exponential factor $e^{-\frac{t}{2}}$ is omitted. We
have given here a corrected formula.}
\begin{equation}\label{conj2 eq}
|\mathcal{Z}|(k,2h) = (-1)^h \frac{G(k+1)^2}{G(2k+1)}\frac{\mathrm{d}^{2h}}{\mathrm{d}t^{2h}}\left[e^{-\frac{t}{2}}\sum_{\substack{\lambda \in  \mathrm{P} \\ \, \lambda_1 \leqslant k}} \prod_{(i,j)\in\lambda} \frac{(2k+i-j)(k+i-j)}{h_\lambda^2(i,j)\left(\lambda_j' - i - j + 1 + 2k\right)^2}t^{|\lambda|}\right].
\end{equation}
\end{conjecture}
We will now prove the following.
\begin{theorem}\label{conj proof}
    Conjecture \ref{conj2} follows from Dehaye's result \cite{kn:dehaye10p}, who proved that
    \begin{equation}\label{dehaye conformal}
        |\mathcal{Z}|(k,2h)=(-1)^h \frac{G(k+1)^2}{G(2k+1)}\sum_{r=0}^{2h}\binom{2h}{r}\left(\frac{-1}{2}\right)^{2h-r}
       r!\sum_{\substack{\\ \lambda\vdash_{k}r}}\frac{ [k]_\lambda}{H^2_\lambda [2k]_\lambda},
    \end{equation}
    for $k,h\in\mb{N}_0$ satisfying $h\le k$.
\end{theorem}
\begin{proof}
We recall that if $A(t)$ and $B(t)$ are (exponential) generating functions given by 
\[A(t)=\sum_{r=0}^\infty a_r\frac{t^r}{r!}\quad\text{and} \quad B(t)=\sum_{r=0}^\infty b_r\frac{t^r}{r!}, \] then the coefficients of the expansion of their product, are given by
\[\frac{\mathrm{d}^n}{\mathrm{d}t^n}\left[A(t)B(t)\right]\big|_{t=0}=\sum_{r=0}^n\binom{n}{r}a_{n-r}b_r.\]
Therefore, Dehaye's result \eqref{dehaye conformal} may be written as
\begin{equation}\label{idk}
|\mathcal{Z}|(k,2h) = (-1)^h \frac{G(k+1)^2}{G(2k+1)}\frac{\mathrm{d}^{2h}}{\mathrm{d}t^{2h}}\left[e^{-\frac{t}{2}}\sum_{\substack{\lambda \in  \mathrm{P} \\ \, \ell(\lambda)\le k}}\frac{ [k]_\lambda}{H^2_\lambda [2k]_\lambda}t^{|\lambda|}\right].\end{equation}
 We may index this sum by the transposes of partitions, rather than the partitions themselves:
    \[\begin{split}\sum_{\substack{\lambda \in  \mathrm{P} \\ \, \ell(\lambda)\le k}}\frac{ [k]_\lambda}{H^2_\lambda [2k]_\lambda}t^{|\lambda|}=\sum_{\substack{ \lambda\in\mathrm{P}\\\lambda_1\le k}}\frac{ [k]_{\lambda'}}{H^2_{\lambda'}[2k]_{\lambda'}}t^{|\lambda|}=\sum_{\substack{ \lambda\in\mathrm{P}\\\lambda_1\le k}}\prod_{(i,j)\in\lambda} \frac{(k+i-j)}{h^2_{\lambda}(i,j)\left(2k+i-j\right)}t^{|\lambda|},
    \end{split}\]
     where we have made use of the fact that $H_\lambda=H_{\lambda'}$ and $|\lambda|=|\lambda'|$.
      Furthermore, we have
    \begin{equation}\label{eq pochammer symbol}
    \begin{split}
        \prod_{(i,j)\in\lambda} (2k+i-j)=\prod_{i=1 }^{\ell(\lambda')}\prod_{j=1}^{\lambda'_i}(2k+j-i)&=\prod_{i=1}^{\ell(\lambda')}\prod_{j=1}^{\lambda'_i}(2k+\lambda'_i-j-i+1),\\
        &=\prod_{(i,j)\in\lambda}\left(\lambda_j'-i-j+1-2k\right),
        \end{split}
    \end{equation}
    where the first equality is justified by reflecting the Young diagram of $\mu$ in the diagonal, and the second follows from mapping $j\mapsto\mu'_i-j+1$. Therefore, we have proven the identity \eqref{conj2 eq} and Conjecture \ref{conj2} follows.
\end{proof}

\subsection{Acknowledgements}
 The author would like to express his gratitude to his supervisors Emma Bailey and Nina Snaith for carefully reading through previous versions of
the paper, and for their continued support. He would also like to thank Harini Desiraju, Sacha Grover and Micha\l \ Szwej for helpful conversations and encouragement. The author gratefully acknowledges support from the Heilbronn Institute for Mathematical Research.

\section{The joint moment problem and Kostka numbers}

In this section, we will prove Theorem \ref{Theorem A}. Our method is inspired by the work of Dehaye on the first order joint moments \cite{kn:dehaye10, kn:dehaye10p}. He proved an exact formula for $\mathcal{M}_N(k,r)$ by relating it to derivatives of Schur polynomials. In Proposition \ref{prop modified moments}, we will prove
\[\mathcal{M}_N(k;\alpha)=\partial^\alpha s_{\langle N^k\rangle}(\{1\}^{2k}),\] 
where for a weak composition $\alpha=(\alpha_1,\ldots,\alpha_{2k})$ we define
\[\partial^\alpha=\overbrace{\partial_{x_1}\,\partial_{x_1}\,\cdots\,\partial_{x_1}}^{\alpha_1\ \text{times}} \ \ \overbrace{\partial_{x_2}\,\partial_{x_2}\,\cdots\,\partial_{x_2}}^{\alpha_2\ \text{times}} \ \bullet\bullet \bullet   \ \overbrace{\partial_{x_{2k}}\,\partial_{x_{2k}}\,\cdots\,\partial_{x_{2k}}}^{\alpha_{2k}\ \text{times}} .\]

The generalised binomial theorem of \cite{OkounkovOlshanski1} will allow us to compute these derivatives. The remainder of the section will then be spent proving Theorem \ref{Theorem 1} and Theorem \ref{Theorem 2}.
\subsection{Computing \texorpdfstring{$\mathcal{M}_N(k;\alpha)$}{M\_N(k; alpha)}}\label{modified moments and schur}
The following lemma was first proven in \cite{kn:cfkrs1}, and then proven again by Bump and Gamburd \cite{kn:bumgam06}.
    \begin{lemma}\label{BG lemma}
    Let $k,N\in \mb{N}$ and $x_1,\ldots,x_{2k}$ be complex numbers. Then, \[\int_{\mathrm{U}(N)}\prod_{j=1}^k\Lambda_{U^\dagger}(x_j^{})\Lambda_U(x_{k+j})\mathrm{d}U=(x_1\cdots x_k)^Ns_{\langle N^k\rangle}(x_1^{-1},\ldots,x_k^{-1},x_{k+1},\ldots ,x_{2k}).\]
\end{lemma}
To identify $\mathcal{M}_N(k;\alpha)$ with the derivatives of Schur polynomials, we will closely follow the proof of Lemma \ref{BG lemma} by Bump and Gamburd. In particular, we will use the \emph{dual Cauchy identity} and Schur orthogonality.
\begin{proposition}[Cauchy identity]
Let $z_{1}, \ldots, z_{p}$ and $y_{1}, \ldots, y_{q}$ be two sets of variables. Then
\begin{equation}\label{cauchy}
\prod_{i=1}^{p} \prod_{j=1}^{q} (1 - z_{i} y_{j})^{-1}
= \sum_{\lambda} s_{\lambda}(z_{1}, \ldots, z_{p}) \,
  s_{\lambda}(y_{1}, \ldots, y_{q}),
\end{equation}
where the sum is over all partitions $\lambda$ of length $\leq \min(p,q)$.

Both sides are convergent provided the $|z_{i}|$ and $|y_{j}|$ are all less than $1$;
otherwise, this is a formal identity.

We also have the \emph{dual Cauchy identity}
\begin{equation}\label{cauchydual}
\prod_{i=1}^{p} \prod_{j=1}^{q} (1 + z_{i} y_{j})
= \sum_{\lambda\leq \langle q^p\rangle} s_{\lambda}(z_{1}, \ldots, z_{p}) \,
  s_{\lambda'}(y_{1}, \ldots, y_{q}).
\end{equation}
\end{proposition}
The importance of the dual Cauchy identity to averages of characteristic polynomials is the following observation. If $U\in \mathrm{U}(q)$, then
\begin{equation}\label{eq16}
    \prod_{i=1}^p\Lambda_U(x_i)=\sum_{\lambda \le  \langle q^p\rangle}(-1)^{|\lambda|}s_\lambda (x_1,\ldots,x_p)\overline{s_{\lambda'}(U)}.
\end{equation}
The following proposition is helpful since the Schur polynomial $s_{\langle N^k\rangle}(x_1,\ldots,x_{2k})$ is a character of an irreducible representation of $\mathrm{U}(2k)$. Consequently, it has many interesting properties, which are of use to us.

\begin{proposition}\label{prop modified moments}
    Let $\alpha$ be a partition of length $\le 2k$. Then
    \[\mathcal{M}_N(k;\alpha)=\partial^\alpha s_{\langle N^k\rangle}(\{1\}^{2k}).\]
\end{proposition}
\begin{proof}
    We have
    \begin{align*}
         \left| \Lambda_U(1)\right|^{2k}&=\Lambda_U(1)^k\det(I-U^\dagger)^k\
        =(-1)^{kN}\left(\det U\right)^k\Lambda_U(1)^{2k}.
    \end{align*}
   If we set $t_1,\ldots,t_N$ to be the eigenvalues of $U$, then 
   \[\left(\det U\right)^k=(t_1\cdots t_N)^k=s_{\langle k^N\rangle}(t_1,\ldots,t_N)=s_{\langle k^N\rangle}(U),\] where the second equality is a simple check from Definition \ref{def schur}. Thus,
   \begin{align*}
       \mathcal{M}_N(k;\alpha)&=\mb{E}_{\mathrm{U}(N)}\left[\left| \Lambda_U(1)\right|^{2k}\prod_{i=1}^{\ell(\alpha)}\left(\frac{\Lambda_U^{(\alpha_i)}(1)}{\Lambda_U(1)}\right)\right]\\
       &=(-1)^{kN}\mb{E}_{\mathrm{U}(N)}\left[\left(\det U\right)^k\left(\Lambda_U(1)\right)^{2k-\ell(\alpha)}\prod_{i=1}^{\ell(\alpha)}{\Lambda_U^{(\alpha_i)}(1)}\right]\\
       &=(-1)^{kN}\mb{E}_{\mathrm{U}(N)}\left[s_{\langle k^N\rangle}(U)\partial^\alpha\prod_{i=1}^{2k}\Lambda_U(x_j)\Big|_{x_1=\cdots =x_{2k}=1}\right],
   \end{align*}
   where we are taking expectation with respect to the Haar measure.
   After setting $q=N$ and $p=2k$, we may apply equation (\ref{eq16}) and obtain
   
   \begin{align*}
       \mathcal{M}_N(k;\alpha)&=(-1)^{kN}\partial^\alpha\mb{E}_{\mathrm{U}(N)}\left[s_{\langle k^N\rangle}(U)\sum_\lambda (-1)^{|\lambda|}s_\lambda(x_1,\ldots,x_{2k})\overline{s_{\lambda'}(U)}\right]\Bigg|_{x_1=\cdots =x_{2k}=1}\\
       &=\sum_\lambda (-1)^{|\lambda|+kN}\partial^\alpha \big(s_\lambda(x_1,\ldots,x_{2k})\cdot\langle\chi_{\langle k^N\rangle},\chi_{\lambda'}\rangle_{\mathrm{U}(N)}\big)\big|_{x_1=\cdots =x_{2k}=1}.
   \end{align*}
   By Schur orthogonality, all terms vanish except the partition $\lambda'=\langle k^N\rangle$, or equivalently, $\lambda=\langle N^k\rangle$.
\end{proof}

\subsubsection{Derivatives of Schur Polynomials and Proof of Theorem \ref{Theorem A}.}
By Proposition \ref{prop modified moments}, we have reduced our task to computing derivatives of Schur polynomials. 
Equivalently, we wish to compute the coefficients in the Taylor expansion
\begin{equation}\label{Schur Taylor}
    s_{\langle N^k\rangle}(1+x_1,\ldots,1+x_{2k})=\sum_{\mu}\frac{\partial ^\mu s_{\langle N^k\rangle}(\{1\}^{2k})}{\mu!}m_\mu(x_1,\ldots,x_{2k}).
\end{equation}
We emphasise that \eqref{Schur Taylor} is indeed the (multivariate) Taylor expansion of $s_{\langle N^k\rangle}(x_1,\ldots,x_{2k})$ about the point $(1,\ldots,1)$, where we have used the symmetry of the Schur polynomial to sum over the symmetric polynomials $m_\mu(x_1,\ldots,x_{2k})$.

However, observe that for a partition $\lambda$,
\[s_\lambda(1+x_1,\ldots,1+x_n)=s_\lambda(x_1,\ldots,x_n)+ \text{lower order terms}.\] Hence, it is more convenient to obtain the expansion in (\ref{Schur Taylor}) by first passing through an expansion in Schur polynomials. Fortunately, such an expansion is known due to Okounkov and Olshanski \cite{OkounkovOlshanski1}.
\begin{definition}\label{shifted def}
    The \textit{shifted Schur polynomial} in $n$ variables, 
indexed by a partition $\mu$ with $\ell(\mu) \leq n$, is defined as the quotient of two $n\times n$ determinants,
\[s^*_\mu(x_1,\ldots x_n)=\frac{\det\left((x_i+n-i)_{\mu_j+n-j}\right)}{\det\left((x_i+n-i)_{n-j}\right)}.\]
\end{definition}
The analogy to the standard Schur polynomials can be seen by comparing Definition \ref{shifted def} with \eqref{schur determinant}.
The following formula, proved by Okounkov and Olshanski, is called the generalised binomial formula. If $n=1$, the formula turns into the classical binomial formula.
\begin{theorem}[Generalised binomial theorem, \cite{OkounkovOlshanski1}]\label{binom}
 Let $\lambda$ be a partition of length $\leq n$. Then \begin{equation}\label{binomial2}
      \begin{split}
        \frac{s_{\lambda}(1+t_{1}, \ldots, 1+t_{n})}
     {s_{\lambda}(\{1\}^{n})}
=&\sum_{\substack{\mu \\ l(\mu) \leq n}}
\frac{s^{*}_{\mu}(\lambda_1,\ldots,\lambda_n) \,
      s_{\mu}(t_{1}, \ldots, t_{n})}
     {[n\boxt{\mu}}.
     \end{split}
    \end{equation}
\end{theorem}
We also need the following lemma, due to Dehaye \cite{kn:dehaye10p}.
\begin{lemma}[Lemma 3 of \cite{kn:dehaye10p}]\label{LemmaDehaye}
    Let $k$ and $N$ be positive whole numbers and let $\mu$ be any partition with length $\le k$. Then, we have \begin{equation}
        s^*_\mu \left(\{N\}^k\right)=(-1)^{|\mu|}\frac{[-N]_\mu\cdot [k]_\mu}{H_\mu}.
    \end{equation}
\end{lemma}
\begin{proof}[Proof of Theorem \ref{Theorem A}]
    By Proposition \ref{prop modified moments}, 
    \[\mathcal{M}_N(k;\alpha)=\partial^\alpha s_{\langle N^k\rangle}(\{1\}^{2k})=\partial^\alpha s_{\langle N^k\rangle}(1+x_1,\ldots,1+x_{2k})\big|_{x_1=\cdots=x_{2k}=0}.\] We set $\lambda=(\lambda_1,\ldots,\lambda_{2k})=(\{N\}^k,\{0\}^k)$ so that
    \[s_\lambda(x_1,\ldots,x_{2k})=s_{\langle N^k\rangle}(x_1,\ldots,x_{2k}).\]
    After setting $n=2k$, we may apply Theorem \ref{binom} to obtain
    \begin{equation}\label{eq binom}
    \partial^\alpha s_{\langle N^k\rangle}(1+x_1,\ldots,1+x_{2k})\big|_{x_1=\cdots=x_{2k}=0}=s_{\langle N^k\rangle}(\{1\}^{2k})\sum_{\substack{\mu \\ l(\mu) \leq 2k}}
\frac{s^{*}_{\mu}(\lambda_1,\ldots,\lambda_{2k}) \,
      \partial^\alpha s_{\mu}(\{0\}^{2k})}
     {[2k\boxt{\mu}}.\end{equation}
     Observe that Lemma \ref{BG lemma} gives
     \[s_{\langle N^k\rangle}(\{1\}^{2k})=\int_{\mathrm{U}(N)}\left|\Lambda_U(1)\right|^{2k}\mathrm{d}U=|\Lambda|_N(k,0).\]
     Also, using the expansion of Schur polynomials into the monomials, Lemma \ref{Kostkalemma}, we have
     \[ \partial^\alpha s_{\mu}(\{0\}^{2k})=\alpha!K_{\mu\alpha}.\] We have used the fact that $\partial^\alpha m_\beta (\{0\}^{2k})=\alpha! \delta_{\beta\alpha}$, where
     \[\delta_{\beta\alpha}\coloneq \begin{cases}
         1 \quad \text{if } \beta=\alpha,\\
         0 \quad \text{otherwise.}
     \end{cases} \]Combining this with \eqref{eq binom} implies,
     \begin{equation}\label{eq modified}
        \mathcal{M}_N(k;\alpha)= \alpha!|\Lambda|_N(k,0)\sum_{\substack{\\ \mu\vdash_{2k}|\alpha|}}K_{\mu\alpha}\frac{ s^*_\mu(\lambda_1,\ldots,\lambda_{2k})}{ [2k]_\mu},
     \end{equation}
     where we may restrict to partitions of size $|\alpha|$ since otherwise $K_{\mu\alpha}=0$. In \cite{OkounkovOlshanski1}, Okounkov and Olshanski proved that shifted Schur functions are (like their standard counterparts) \emph{stable} in the sense that for any partition $\mu$ and for each $n$,
     \[s_\mu^*(x_1,\ldots,x_n,0)=s_\mu^*(x_1,\ldots,x_n).\] Therefore, by Lemma \ref{LemmaDehaye},
     \[s^*_\mu(\lambda_1,\ldots,\lambda_{2k})=  s^*_\mu \left(\{N\}^k\right)=(-1)^{|\mu|}\frac{[-N]_\mu\cdot [k]_\mu}{H_\mu}.\]
If $\ell(\mu)>k,$ then by definition, $s^*_\mu \left(\{N\}^k\right)=0$. Thus, we only sum over partitions $\mu\vdash_k~\!\!\!|\alpha|.$
\end{proof}
\begin{proof}[Proof of Corollary \ref{Cor A}]
    Observe that for a partition $\mu$ of size $|\alpha|$,
    \[[-N]_\mu=\prod_{i=1}^{\ell (\mu)}\prod_{j=1}^{\mu_i}(-N+j-i)\sim (-N)^{|\alpha|} ,\]
    as $N\to\infty$. Therefore, 
      \begin{align*}
         \mathcal{M}_N(k;\alpha)&=\alpha!(-1)^{|\alpha|}|\Lambda|_N(k,0)\sum_{\substack{\\ \mu\vdash_{k}|\alpha|}}K_{\mu\alpha}\frac{ [-N]_\mu[k]_\mu}{H_\mu [2k]_\mu}\\
         &\sim \alpha!|\Lambda|_N(k,0)N^{|\alpha|}\sum_{\substack{\\ \mu\vdash_{k}|\alpha|}}K_{\mu\alpha}\frac{ [k]_\mu}{H_\mu [2k]_\mu}.
     \end{align*}
     In \cite{kn:keasna00a}, Keating and Snaith proved that for $k\in\mathbb{N}$, 
     \begin{equation}\label{KS asymptotic}
     |\Lambda|_N(k,0)\sim N^{k^2}\prod_{j=0}^{k-1}\frac{j!}{(j+k)!},
     \end{equation}and thus, we conclude the result.
\end{proof}

\subsection{Basic Relations}\label{basic relations}
We now wish to prove Theorem \ref{Theorem 1} and Theorem \ref{Theorem 2}. In fact, in Proposition \ref{linear lambda} and Proposition \ref{linear Z}, we will prove something stronger. That is, the linear relationship between the different types of joint moments follows from the fact that $\frac{\Lambda_{U^\dagger}^{(n)}(1)}{\Lambda_{U^\dagger}(1)}$ may be written as a linear combination of $\frac{\Lambda_{U}^{(l)}(1)}{\Lambda_{U}(1)}$, for $l \le n$. As in \cite{kn:dehaye10,kn:dehaye10p}, we call these the `basic relations'. In particular, these will generalise the relation
\begin{equation}
    \begin{split}
        \left|\frac{\mathcal{Z}_U'(0)}{\mathcal{Z}_U{(0)}}\right|^2
        &=-\frac{N^2}{4}+N\frac{\Lambda'_U(1)}{{\Lambda}_U(1)}-\left(\frac{\Lambda'_U(1)}{{\Lambda}_U(1)}\right)^2,
    \end{split}
\end{equation}
to higher order derivatives.
\subsubsection{Proof of Theorem \ref{Theorem 1}.}
\begin{proposition}\label{linear lambda}
    Let $n$ be a non-negative integer. Then, for $U\in \mathrm{U}(N)$,
    \[\frac{\Lambda_{U^\dagger}^{(n)}(1)}{\Lambda_{U^\dagger}(1)}=\sum_{j=0}^n (-1)^j \binom{n}{j}\left(N\right)_{n-j}\sum_{l\le j}L(j,l)\frac{\Lambda^{(l)}_U(1)}{\Lambda_U(1)},\]
    where $L(j,l)=\binom{j-1}{l-1}\frac{j!}{l!}$ are the Lah numbers.
\end{proposition}
\begin{proof}
    For every $x\in\mb{C}\setminus\{0\},$
    \[\frac{\Lambda_{U^\dagger}^{}(x)}{\Lambda_{U^\dagger}(1)}=\frac{\det(I-xU)}{\det(I-U)}=x^N\frac{\Lambda_U(x^{-1})}{\Lambda_U(1)}.\] Therefore, by the Leibniz rule,
    \[\frac{\Lambda_{U^\dagger}^{(n)}(1)}{\Lambda_{U^\dagger}(1)}=\sum_{j=0}^n\binom{n}{j}\left(N\right)_{n-j}\frac{\mathrm{d}^j}{\mathrm{d}x^j}\frac{\Lambda^{}_U(x^{-1})}{\Lambda_U(1)}\Big|_{x=1}.\]
    We set $u=1/x$, and thus $\frac{\mathrm{d}^j}{\mathrm{d}x^j}=(-u^2\frac{\mathrm{d}}{\mathrm{d}u})^j$. The Lah numbers satisfy the recursive formula
\[L(n+1,m)=(n+m)L(n,m)+L(n,m-1).\]
  Therefore, it is easy to verify that  
    \[\left(u^2\frac{\mathrm{d}}{\mathrm{d}u}\right)^j=\sum^j_{l=1}L(j,l)u^{j+l}\frac{\mathrm{d}^l}{\mathrm{d}u^l},\] 
    and hence, 
    \[\frac{\mathrm{d}^j}{\mathrm{d}x^j}\frac{\Lambda^{}_U(x^{-1})}{\Lambda_U(1)}\Big|_{x=1}=(-1)^j\sum_{l\le j}L(j,l)\frac{\Lambda^{(l)}_U(1)}{\Lambda_U(1)},\]
    which gives the result.
\end{proof}
\begin{proof}[Proof of Theorem \ref{Theorem 1}]
By Proposition \ref{linear lambda},
    \begin{equation}\label{eq linear lambda}        \begin{split}\prod_{i=1}^{\ell(\alpha)}\left|\frac{\Lambda_U^{(\alpha_i)}(1)}{\Lambda_U(1)}\right|^2&=\prod_{i=1}^{\ell(\alpha)}\frac{\Lambda_U^{(\alpha_i)}(1)}{\Lambda_U(1)}\frac{\Lambda_{U{^\dagger}}^{(\alpha_i)}(1)}{\Lambda_{U^\dagger}(1)}\\
    &=\prod_{i=1}^{\ell(\alpha)}\frac{\Lambda_U^{(\alpha_i)}(1)}{\Lambda_U(1)}\sum_{\substack{\beta_i\leq \alpha_i\\ \gamma_i\leq \beta_i}}(N)_{\alpha_i-\beta_i}(-1)^{\beta_i} \binom{\alpha_i}{\beta_i}L(\beta_i,\gamma_i)\frac{\Lambda_U^{(\gamma_i)}(1)}{\Lambda_U(1)}\\
        &=\sum_{\substack{\beta\leq \alpha\\ \gamma\leq \beta}}(N)_{\alpha-\beta}(-1)^{|\beta|} \binom{\alpha}{\beta}L(\beta,\gamma)\prod_{i=1}^{\ell(\alpha)}\left(\frac{\Lambda_U^{(\alpha_i)}(1)}{\Lambda_U(1)}\frac{\Lambda_U^{(\gamma_i)}(1)}{\Lambda_U(1)}\right).
        \end{split}
    \end{equation}
    Therefore, by the linearity of expectation, we have
    \[ {|\Lambda|_N(k;2\!\cdot\!\alpha)}=\sum_{\substack{\beta\leq \alpha\\ \gamma\leq \beta}}(N)_{\alpha-\beta}(-1)^{|\beta|} \binom{\alpha}{\beta}L(\beta,\gamma)\mathcal{M}_N(k;\alpha\!\cup\!\gamma),\]
    which proves the first part of Theorem \ref{Theorem 1}.
    
    For the large $N$ limit, since $(N)_{\alpha-\beta}=N^{|\alpha|-|\beta|}\left(1+O\left(\frac{1}{N}\right)\right),$ where the implied constant is dependent on $\alpha$ and $\beta$,
    \[   \frac{|\Lambda|_N(k;2\!\cdot\!\alpha)}{N^{k^2+2|\alpha|}}=\sum_{\substack{\beta\leq \alpha\\ \gamma\leq \beta}}(-1)^{|\beta|} \binom{\alpha}{\beta}L(\beta,\gamma) \frac{\mathcal{M}_N(k;\alpha\!\cup\!\gamma)}{{N^{k^2+|\alpha\cup\beta|}}}\left(1+O\left(\frac{1}{N}\right)\right).\]
    It has been proven in Corollary \ref{Cor A} that for any fixed weak composition $\delta$, the correct scaling in $N$ for $\mathcal{M}_N(k;\delta)$ is $N^{k^2+|\delta|}$. In particular, if $\gamma \lneq \beta$, then $|\gamma|\le |\beta|-1$, and therefore,
    \[\frac{\mathcal{M}_N(k;\alpha\!\cup\!\gamma)}{{N^{k^2+|\alpha\cup\beta|}}}=\frac{1}{N^{|\beta|-|\gamma|}}\frac{\mathcal{M}_N(k;\alpha\!\cup\!\gamma)}{{N^{k^2+|\alpha\!\cup\!\gamma|}}}=O\left(\frac{1}{N}\right).\]
    Hence, the weak compositions $\gamma\lneq\beta$ will not contribute to the large $N$ limit
    and we conclude
    \[ {|\Lambda|(k;2\!\cdot\!\alpha)}=\sum_{\substack{\beta\leq \alpha\\ }}(-1)^{|\beta|}\binom{\alpha}{\beta}\mathcal{M}(k;\alpha\!\cup\!\beta).\]
\end{proof}
\subsubsection{Proof of Theorem \ref{Theorem 2}.}
To prove Theorem \ref{Theorem 2}, we follow a similar method to the proof of Theorem \ref{Theorem 1}.
Recall from (\ref{Z-polynomial}), that we define
\[\mathcal{Z}_U(\theta)\coloneq A_U(\theta)\Lambda_U(e^{\mf{i}\theta}),\]
where $A_U(\theta)\coloneq e^{-\mf{i}N(\theta+\pi) /2}e^{\mf{i}\sum_{n=1}^N\theta_n/2}$, and that $\mathcal{Z}_U(\theta)\in\mathbb{R}$ for $\theta\in\mb{R}$.
\begin{proposition}\label{linear Z}
    Let $n$ be a non-negative integer. Then
   \[\frac{\mathcal{Z}_U^{(n)}(0)}{\mathcal{Z}_U(0)}=\mf{i}^n\sum_{j=0}^n \binom{n}{j}\left(\frac{-N}{2}\right)^{n-j}\sum_{l\le j}{j\brace l}
   \frac{\Lambda^{(l)}_U(1)}{\Lambda^{}_U(1)},\]
   where ${j\brace l}$ are the Stirling numbers of the second kind.
\end{proposition}
\begin{proof}
   For each non-negative integer $m$, we have
    \[A_U^{(m)}(0)=\left(\frac{-\mf{i}N}{2}\right)^{m}A_U(0).\]
    Thus, by the Leibniz rule,
    \[{\mathcal{Z}_U^{(n)}(0)}{}=\sum_{j=0}^n \binom{n}{j}\left(\frac{-\mf{i}N}{2}\right)^{n-j}A_U(0)\frac{\mathrm{d}^j}{\mathrm{d}\theta^j}\Lambda^{}_U(e^{\mf{i}\theta})\Big|_{\theta=0}.\] 
    We divide by $\mathcal{Z}_U(0)=A_U(0)\Lambda_U(1)$ so that
     \[\frac{\mathcal{Z}_U^{(n)}(0)}{\mathcal{Z}_U(0)}=\sum_{j=0}^n \binom{n}{j}\left(\frac{-\mf{i}N}{2}\right)^{n-j}\frac{\frac{\mathrm{d}^j}{\mathrm{d}\theta^j}\Lambda^{}_U(e^{\mf{i}\theta})\big|_{\theta=0}}{\Lambda^{}_U(1)}.\] 
     We now proceed as in the proof of Proposition \ref{linear lambda}, except we use the substitution \(u=e^{\mf{i}\theta}\), so that  $\frac{\mathrm{d}}{\mathrm{d}\theta}=\mf{i}u\frac{\mathrm{d}}{\mathrm{d}u}$, and the relation 
     \[\left(u\frac{\mathrm{d}}{\mathrm{d}u}\right)^n=
\sum^n_{m=1}{n\brace m} u^{m}\frac{\mathrm{d^m}}{\mathrm{d}u^m}.\]
    
\end{proof}
\begin{proof}[Proof of Theorem \ref{Theorem 2}]
    By Proposition \ref{linear Z},
    \begin{equation}
    \begin{split}\prod_{i=1}^{\ell(\alpha)}\left|\frac{\mathcal{Z}_U^{(\alpha_i)}(0)}{\mathcal{Z}_U(0)}\right|^2&=\prod_{i=1}^{\ell(\alpha)}\left(\frac{\mathcal{Z}_U^{(\alpha_i)}(0)}{\mathcal{Z}_U(0)}\right)^2\\
    &=(-1)^{|\alpha|}\sum_{\substack{\beta\leq 2\cdot\alpha\\ \gamma\leq \beta}}\left(\frac{-N}{2}\right)^{2|\alpha|-|\beta|}\binom{2\!\cdot\!\alpha}{\beta}{\beta\brace \gamma}\prod_{i=1}^{\ell(\alpha)} \frac{\Lambda_U^{(\gamma_i)}(1)}{\Lambda_U(1)},
    \end{split}
    \end{equation}
    where the second equality follows from considering the expression as a product over the weak composition $2\!\cdot\!\alpha$. After taking expectations, we obtain
    \[ {|\mathcal{Z}|_N(k;2\!\cdot\!\alpha)}=(-1)^{|\alpha|}\sum_{\substack{\beta\leq 2\cdot\alpha\\ \gamma\leq \beta}}\left(\frac{-N}{2}\right)^{2|\alpha|-|\beta|}\binom{2\!\cdot\!\alpha}{\beta}{\beta\brace \gamma}\mathcal{M}_N(k;\gamma).\]
    which proves the first part. To prove that
    \[  {|\mathcal{Z}|(k;2\!\cdot\!\alpha)}=(-1)^{|\alpha|}\sum_{\substack{\beta\leq 2\cdot\alpha\\ }}\left(\frac{-1}{2}\right)^{2|\alpha|-|\beta|}\binom{2\!\cdot\!\alpha}{\beta}\mathcal{M}(k;\beta),\]
    we observe that terms in the sum for which $\gamma \lneq \beta$ vanish in the large $N$ limit after dividing both sides by $N^{k^2+2|\alpha|}$ and applying the same reasoning as used in the proof of Theorem \ref{Theorem 1}.
\end{proof}

\section{The Jacobi-Trudi formula and its applications}\label{section 3}

    The purpose of this section is to highlight some applications of the Jacobi-Trudi identity to averages of characteristic polynomials over the unitary group. The first application is a proof of Theorem \ref{Theorem B}, which is in turn used to reprove the formulas provided in \cite{kn:keawei1}. 
    
    The second application is to give a short proof of the contour integral formulation of averages of characteristic polynomials given in \cite{kn:crs06}, which is our Lemma \ref{contour products}. This lemma is often the starting point for moment calculations in random matrix theory over the unitary group. We prove that it arises naturally as a consequence of a contour integral expression for the complete homogeneous symmetric polynomials. 
    \begin{subsection}{Kostka numbers and magic squares: proof of Theorem \ref{Theorem B}}
        A polynomial $f\in \mb{Q}[t_1,\ldots ,t_k]$ is \emph{skew-symmetric} if \[f(t_{\sigma(1)},\ldots ,t_{\sigma(k)})=\operatorname{sgn}(\sigma)f(t_1,\ldots ,t_k),\]
        for all $\sigma\in S_k$.
We say that a function is a \emph{shifted symmetric polynomial} if it is symmetric with respect to the shifted variables 
\[t'_i\coloneq t_i+const.-i,\] for some arbitrary constant. Without loss of generality, we may take the constant to be zero.
A canonical example of a polynomial which is skew-symmetric in the shifted variables $t'_i=t_i-i$, is the shifted Vandermonde
\[\Delta'(t_1,\ldots ,t_n)\coloneq \prod_{i<j}(t_j-t_i+i-j).\]
We may also consider the \emph{field of fractions} of $\mb{Q}[t_1,\ldots t_k]$, given by
\[\mb{Q}(t_1,\ldots t_k)\coloneq \left\{\frac{f(t_1,\ldots,t_k)}{g(t_1,\ldots,t_k)} \ : \ f,g\in \mb{Q}[t_1,\ldots t_k], \ g\not=0 \right\}.\]

If a function $f\in\mb{Q}(t_1,\ldots t_k)$ is skew-symmetric in the shifted variables $t'_i$, then it vanishes whenever $t'_i=t'_j$. Therefore, we may write
\begin{equation}
    f(t_1,\ldots ,t_k)=\prod_{i<j}(t_j-t_i+i-j) \cdot g(t_1,\ldots ,t_k),
\end{equation}
where $g\in\mb{Q}(t_1,\ldots t_k)$ is symmetric in the shifted variables.

We will prove Theorem \ref{Theorem B} from Theorem \ref{Theorem A} as an application of the following lemma.

 \begin{lemma}\label{Lemma KostkaMagic} Let $n$ and $k$ be non-negative integers, which satisfy $k\le n$.
            Let $f\in\mb{Q}(t_1,\ldots ,t_k)$ be a skew-symmetric function in the \emph{shifted} variables $t'_i=t_i-i$. Furthermore, assume that $f$ has no poles at the points $(t_1,\ldots,t_k)\in\mb{N}_0^k$ that satisfy $\sum_{i=1}^{k}t_i=n$.  Then,
            \begin{eqnarray}
                \sum_{\mu\vdash_k{n}}K_{\mu\lambda} f(\mu_1,\ldots ,\mu_k)&=&\sum_{\substack{\sum_{i=1}^kt_i=n\\t_i\ge 0}}N_{\lambda,t}f(t_1,\ldots,t_k),
            \end{eqnarray}
               where the first sum is over partitions $\mu=(\mu_1,\ldots,\mu_k)$ and the second sum is over weak compositions $t=(t_1,\ldots ,t_k)$.
        \end{lemma}
        To prove Lemma \ref{Lemma KostkaMagic}, we need the following theorem, which is known to experts (see, for example, \cite{kn:lederer}). It is equivalent to Theorem \ref{JacobiTrudi}, the Jacobi-Trudi formula.
        \begin{theorem}\label{TheoremKostkaMagic}
           Let $\lambda$ be a partition of length $k$ and $\beta$ be any weak composition. Then,
            \begin{equation}
                K_{\lambda \beta} =\sum_{\sigma\in S_k}\operatorname{sgn}(\sigma) N_{\sigma*\lambda,\beta},
            \end{equation}
            where for a partition $\lambda$ of length $k$ and a permutation $\sigma\in S_k$, we define
            \[\sigma*\lambda=\left(\lambda_1-1 +\sigma(1),\ldots ,\lambda_k-k+\sigma(k)\right).\]
        \end{theorem}
        \begin{proof}
       By Theorem \ref{JacobiTrudi},
            \[\begin{split}
            s_\lambda(x_1,\ldots ,x_k)&=\det_{k\times k}(h_{\lambda_i-i+j}(x_1,\ldots ,x_k))\\
            &=\sum_{\sigma\in S_k}\operatorname{sgn}(\sigma)\prod_{i=1}^kh_{\lambda_i-i+\sigma(i)}(x_1,\ldots ,x_k)=\sum_{\sigma\in S_k}\operatorname{sgn}(\sigma) h_{\sigma*\lambda}(x_1,\ldots ,x_k).
            \end{split}\] By comparing the expansions given in Lemma \ref{Kostkalemma} of the Schur polynomials and the complete homogeneous symmetric polynomials into monomials, we conclude
            \[  K_{\lambda \beta} =\sum_{\sigma\in S_k}\operatorname{sgn}(\sigma) N_{\sigma*\lambda,\beta.}\]
        \end{proof}
        We will now prove Lemma \ref{Lemma KostkaMagic}.
        \begin{proof}[Proof of Lemma \ref{Lemma KostkaMagic}]
            We first observe that, since $f$ is skew-symmetric, the only terms in the sum,
            \[\sum_{\substack{\sum_{i=1}^kt_i=n\\t_i\ge 0}}N_{t,\lambda}f(t_1,\ldots,t_k),\] which contribute are those for which $\{t_1-1,\ldots ,t_k-k\}$ are all distinct. Thus, there exists a unique $\sigma\in S_k$ such that
            \[t_{\sigma(1)}-\sigma(1)>\cdots>t_{\sigma(k)}-\sigma(k)\ge -k.\] Therefore, after setting $\mu_i=t_{\sigma(i)}-\sigma(i)+i$, we obtain a partition $\mu\vdash_kn$. Since $\sigma$ and $\mu$ were uniquely determined by $t=(t_1,\ldots ,t_k)$, we have defined an injective mapping:
            \[\{t=(t_1,\ldots ,t_k):|t|=n \ \text{and} \ t_i-i\not=t_j-j \text{ for } i\not=j\}\hookrightarrow (S_k,\mathrm{P}_{n,k}),\]
            where $\mathrm{P}_{n,k}$ denotes the set of partitions of size $n$ and length $\le k.$
            Moreover, the inverse of this mapping is given by 
            \[(\sigma,\mu)\mapsto t=(\mu_{\sigma^{-1}(1)}-\sigma^{-1}(1) + 1,\ldots ,\mu_{\sigma^{-1}(k)}-\sigma^{-1}(k)+k),\] which is a well-defined weak composition as long as $\mu_{\sigma^{-1}(i)}-\sigma^{-1}(i) + i\ge 0$ for all $1\leq i \leq k$. Therefore, we have
            \begin{eqnarray}
                &&\sum_{\substack{\sum_{i=1}^kt_i=n\\t_i\ge 0}}N_{t,\lambda}f(t_1,\ldots,t_k)\notag\\
                &&=\sum_{\mu\vdash_k n}\sum_{\substack{\sigma\in S_k \\ \mu_i\ge i-\sigma(i)}}N_{\sigma*\mu,\lambda}f(\mu_{\sigma^{-1}(1)}-\sigma^{-1}(1) + 1,\ldots ,\mu_{\sigma^{-1}(k)}-\sigma^{-1}(k)+k)\notag\\
                &&=\sum_{\mu\vdash_k n}\left(\sum_{\substack{\sigma\in S_k \\ \mu_i\ge i-\sigma(i)}}\operatorname{sgn}(\sigma)N_{\sigma*\mu,\lambda}\right)f(\mu_1,\ldots ,\mu_k),
            \end{eqnarray}
            where we have used the fact that $f$ is skew-symmetric in the shifted variables for the second equality. 
            
            Clearly, $N_{\sigma*\mu,\lambda}$ is zero if $\mu_i-i+\sigma(i)<0$, and thus the restriction is redundant. In particular, by Theorem \ref{TheoremKostkaMagic},
            \[\sum_{\substack{\sigma\in S_k \\ \mu_i\ge i-\sigma(i)}}\operatorname{sgn}(\sigma)N_{\sigma*\mu,\lambda}=\sum_{\substack{\sigma\in S_k \\ }}\operatorname{sgn}(\sigma)N_{\sigma*\mu,\lambda}=K_{\mu\lambda}.\] \end{proof}
            \begin{proof}[Proof of Theorem \ref{Theorem B}]
            We recall from \eqref{Weyl dimension} and \eqref{Hookeq} that we have the relation
            \[s_\mu(\{1\}^k)=\frac{\prod_{i<j}(\mu_i-\mu_j-i+j)}{\prod_{i}(n-i)!}=\frac{[k]_\mu}{H_\mu}.\]
                Therefore, we may rewrite Theorem \ref{Theorem A} as:
                \begin{equation}
                    \mathcal{M}_N(k;\alpha)=\alpha!(-1)^{|\alpha|}|\Lambda|_N(k,0)\sum_{\substack{\\ \mu\vdash_{k}|\alpha|}}K_{\mu\alpha}\frac{ [-N]_\mu}{[2k]_\mu}\frac{\prod_{i<j}(\mu_i-\mu_j+j-i)}{\prod_{i=1}^k(k-i)!}.
                \end{equation}
                Furthermore, for all $x\in\mb{C}$,
\[[x]_\mu=\prod_{i=1}^{\ell(\mu)}\frac{(\mu_i+x-i)!}{(x-i)!},\]
is symmetric in the shifted variables $\mu_i-i$. We set
\[f(t_1,\ldots ,t_k)=\frac{ [-N]_t}{[2k]_t}\frac{\prod_{i<j}(t_i-t_j+j-i)}{\prod_{i=1}^k(k-i)!}\in\mb{Q}(t_1,\ldots,t_k),\] and observe that $[2k]_t$ is never zero for all weak compositions $t=(t_1,\ldots,t_k)$ of length $\le k$. Hence,
we may apply Lemma \ref{Lemma KostkaMagic} to obtain Theorem \ref{Theorem B}.
            \end{proof}
\begin{proof}[Proof of Corollary \ref{Cor B}]
    Following the same reasoning as above, but with Lemma \ref{Lemma KostkaMagic} applied to Corollary \ref{Cor A}, we have
    \begin{equation}\label{3.6}
        \mathcal{M}(k;\alpha)
         =\alpha!|\Lambda|(k,0)\sum_{\substack{\mf{m}=(m_1,\ldots ,m_k)\\ \sum_{i=1}^km_i=|\alpha|}}N_{\alpha,\mathfrak{m}}\prod_{i=1}^k\frac{(2k-i)!}{(k-i)!}\frac{\prod_{1\le i<j\le k}(m_i-m_j+j-i)}{\prod_{i=1}^k(m_i+2k-i)!}.
    \end{equation}
    We recall the asymptotic relation in \eqref{KS asymptotic}, due to Keating and Snaith \cite{kn:keasna00a}, which states that
    \[|\Lambda|(k,0)=\prod_{i=0}^{k-1}\frac{i!}{(i+k)!}=\prod_{i=1}^k\frac{(k-i)!}{(2k-i)!}.\] After cancellation of these factors in \eqref{3.6}, we have
    \[  \mathcal{M}(k;\alpha)
         =\alpha!\sum_{\substack{\mf{m}=(m_1,\ldots ,m_k)\\ \sum_{i=1}^km_i=|\alpha|}}N_{\alpha,\mf{m}}\frac{\prod_{1\le i<j\le k}(m_i-m_j+j-i)}{\prod_{i=1}^k(m_i+2k-i)!},\]
         as desired.
\end{proof}
    \end{subsection} 
\subsubsection{Comparison with Keating-Wei}\label{section compare}
As mentioned in Subsection \ref{unifying section}, Keating and Wei \cite{kn:keawei1}, provided two alternative formulations for $|\mathcal{Z}|(k;2\cdot\alpha)$, where $\alpha$ is a partition of at most two distinct parts. The first type (see Theorems 1.6 and 3.7 in \cite{kn:keawei1}) is the more computationally efficient expression.

The following theorem is a generalisation of these results in that they are valid for all partitions $\alpha$ of length $\le k$. Moreover, the Keating-Wei formulas follow immediately. The theorem is obtained by writing $N_{\beta,\mf{m}}$ as a sum over matrices with non-negative integer entries.
\begin{theorem}\label{Keating-Wei generalisation} Let $\alpha$ be a partition of length $h\le k$. Then
    \begin{equation}
    \begin{split}
         {|\mathcal{Z}|(k;2\!\cdot\!\alpha)}=(-1)^{|\alpha|}(\alpha!)^2\sum_{\substack{\sum_{l=1}^k m_{q,l}\le2\cdot\alpha_q\\ q=1,\ldots ,2h}}&\left(\frac{-1}{2}\right)^{2|\alpha|-\sum_{q=1}^{2h}\sum_{l=1}^k m_{q,l}}\prod_{q=1}^{2h}\frac{1}{(2\!\cdot\!\alpha_q-\sum_{l=1}^k m_{q,l})!}\\
         &\times\frac{\prod_{1\le i<j\le k}(\sum_{q=1}^{2h}m_{q,i}-\sum_{q=1}^{2h}m_{q,j}+j-i)}{\prod_{i=1}^k(\sum_{q=1}^{2h}m_{q,i}+2k-i)!},
         \end{split}
    \end{equation}    
    where $2\!\cdot\!\alpha_q$ is the $q$th part of the partition $2\!\cdot\!\alpha$.
\end{theorem}
\begin{proof}
Theorem \ref{Theorem 2} allows us to express $ {|\mathcal{Z}|(k;2\!\cdot\!\alpha)}$ in terms of the averages $ {\mathcal{M}(k;\beta)}$, where $\beta\le 2\!\cdot\!\alpha$. Combining this with Corollary \ref{Cor B}, gives the explicit relation
    \begin{equation}\label{4.8}
        \begin{split}
              {|\mathcal{Z}|(k;2\!\cdot\!\alpha)}&=(-1)^{|\alpha|}\sum_{\substack{\beta\leq 2\cdot\alpha\\ }}\left(\frac{-1}{2}\right)^{2|\alpha|-|\beta|}\frac{(\alpha!)^2}{(2\!\cdot\!\alpha-\beta)!}\\
              &\times\sum_{\substack{\mf{m}=(m_1,\ldots ,m_k)\\ }}N_{\beta,\mf{m}}\frac{\prod_{1\le i<j\le k}(m_i-m_j+j-i)}{\prod_{i=1}^k(m_i+2k-i)!}.
        \end{split}
    \end{equation}
 By writing $N_{\beta,\mf{m}}$ as a sum over $2h\times k$ matrices $(m_{i,j})$, we obtain
 \begin{equation}\nonumber
     \begin{split}
         &\sum_{\substack{\mf{m}=(m_1,\ldots ,m_k)\\ }}N_{\beta,m}\frac{\prod_{1\le i<j\le k}(m_i-m_j+j-i)}{\prod_{i=1}^k(m_i+2k-i)!}\\
         &=\sum_{\substack{\sum_{l=1}^k m_{q,l}=\beta_q\\ q=1,\ldots ,2h}}\frac{\prod_{1\le i<j\le k}(\sum_{q=1}^{2h}m_{q,i}-\sum_{q=1}^{2h}m_{q,j}+j-i)}{\prod_{i=1}^k(\sum_{q=1}^{2h}m_{q,i}+2k-i)!}.
     \end{split}
 \end{equation}
 Substituting this into \eqref{4.8} and changing the order of summation gives the result.
\end{proof}
The alternative formulas that Keating and Wei provide (Theorems 1.2 and 3.4 in \cite{kn:keawei1}) are sums over determinants of matrices whose entries involve
the modified Bessel functions of the first kind. The interest in  determinants of such matrices is that they can be used to build connections between
joint moments and Painlev\'e transcendents. This link is established in \cite{kn:keawei2}.

Their formulas involving the modified Bessel functions of the first kind are much longer. Unfortunately, it is unclear how one directly shows that these formulas are equivalent to Theorem \ref{Keating-Wei generalisation}. In the simplest (non-trivial) case, we have $\alpha=\langle1^h\rangle$ and we recover Theorem \ref{crs 1}. The next case we consider is when $\alpha=\langle2^h\rangle$, and so
\[  |\mathcal{Z}|(k;2\!\cdot\!\alpha)=\lim_{N\to\infty}\frac{1}{N^{k^2+4h}}\int_{\mathrm{U}(N)}\left|\mathcal{Z}_U(0)\right|^{2k-2h}\left|\mathcal{Z}_U''(0)\right|^{2h}\mathrm{d}U.\]
The following theorem uses Theorem \ref{Keating-Wei generalisation} to express $|\mathcal{Z}|(k;2\!\cdot\!\alpha)$ in terms of the modified Bessel function, where $\alpha=\langle2^h\rangle$.
\begin{theorem}[Theorem 1.4 of \cite{kn:keawei1}]\label{KW Theorem 4}
   Let $h,k\in\mb{N}_0$ with $h\le k$ and set $\alpha=\langle2^{h}\rangle$. Then,
   \begin{equation}\label{KW1}
       \begin{split}
          |\mathcal{Z}|(k;2\!\cdot\!\alpha)=(-1)^{\frac{k(k-1)}{2}}\sum_{l=0}^{2h} \binom{2h}{l}\left(\frac{\mathrm{d}}{\mathrm{d}x}\right)^{4h-2l}\left(e^{-\frac{x}{2}}x^{-l-\frac{k^2}{2}}g_l(x)\right)\Bigg|_{x=0},
       \end{split}
   \end{equation}
   where
   \[g_l(x)\coloneq \sum_{\substack{\sum_{i=1}^kl_i=l\\ l_i\ge 0}}\binom{l}{l_1,\ldots , l_k}\det(I_{i+j+1+2l_{j+1}}(2\sqrt{x}))_{i,j=0,\ldots ,k-1}.\]
   Moreover,
    \begin{equation}\label{KW2}
\begin{split}
|\mathcal{Z}|(k;2\!\cdot\!\alpha) = \ &  \sum_{\substack{l_0 + \dots + l_k = 2h \\ l_i\ge0}} \sum_{\substack{t_0 + \dots + t_k = 2l_0 \\ t_i\ge0}} (-2)^{-t_0} \binom{2h}{l_0, \dots, l_k} \binom{2l_0}{t_0, \dots, t_k} \\
& \times \left( \prod_{i=1}^{k} \frac{1}{(2k - i + 2l_i + t_i)!} \right) \left( \prod_{1 \le i < j \le k} (2l_i - 2l_j + t_i - t_j - i + j) \right).
\end{split}
\end{equation}
\end{theorem}
\begin{proof}
We will only prove \eqref{KW2} and this will be left to the appendix (see Proposition \ref{appendix prop}). The equivalence of \eqref{KW1} and \eqref{KW2} follows from using the Taylor expansion of the Bessel function. We refer the reader to \cite{kn:crs06, kn:keawei1} for a proof of this fact.
\end{proof}

\subsection{A contour integral variation of the Jacobi-Trudi formula}
The formulas provided in \cite{kn:keawei1}, and many other expressions in the literature on the joint moments, make use of a contour integral expression for averages of characteristic polynomials. We will give a new derivation of this expression (our Lemma \ref{contour products}) using an application of Lemma \ref{BG lemma} and a variation of the Jacobi-Trudi formula. In particular, our methods are based on classical theorems of representation theory. 

The following lemma is based on a contour integral expression for the complete homogeneous symmetric polynomials. We believe that this formulation of the Jacobi-Trudi formula is original.
\begin{lemma}[Jacobi-Trudi]\label{contour JC}
     Let $z(x)=\frac{1}{1-e^{-x}}$, $\lambda$ be a partition of length $\leq n$ and $\alpha_1,\ldots,\alpha_n$ be complex numbers. Then
    \[s_\lambda(e^{\alpha_1},\ldots ,e^{\alpha_n})=\det_{\ell(\lambda)\times \ell(\lambda)}\left(\frac{1}{2\pi\mf{i}} \oint_{\mathcal{C}}e^{(\lambda_i-i+j)w}\prod_{i=1}^nz(w-\alpha_i)\mathrm{d}w\right),\] where $\mathcal{C}$ is any contour enclosing the $\alpha_i$ and $0$.
    
\end{lemma}
\begin{proof}
We will prove that for $k\in\mb{N}$,
\begin{equation}\label{complete homo contour}
    h_k(e^{\alpha_1},\ldots,e^{\alpha_n})=\frac{1}{2\pi\mf{i}} \oint_{\mathcal{C}}e^{kw}\prod_{i=1}^nz(w-\alpha_i)\mathrm{d}w.
\end{equation}
Since $z(x)$ is holomorphic except for a simple pole at $x=0$,
by the residue theorem we have
\[\frac{1}{2\pi\mf{i}} \oint_{\mathcal{C}}e^{kw}\prod_{i=1}^nz(w-\alpha_i)\mathrm{d}w=\sum_{j=1}^n\lim_{w\to\alpha_j}(w-\alpha_j)e^{kw}\prod_{i=1}^nz(w-\alpha_i)=\sum_{j=1}^n\frac{e^{k\alpha_j}}{\prod_{\substack{1\le i\le n\\i\not=j}}1-e^{\alpha_i-\alpha_j}}.\]
We set $x_j=e^{\alpha_j}$ so that, after clearing denominators,
\begin{equation}\label{3.18}
\frac{1}{2\pi\mf{i}} \oint_{\mathcal{C}}e^{kw}\prod_{i=1}^nz(w-\alpha_j)\mathrm{d}w=\sum_{j=1}^n\frac{x_j^{k+n-1}}{\prod_{\substack{1\le i\le n\\i\not=j}}(x_j-x_i)}.\end{equation}
Given that for each $1\le j\le n$, 
\begin{equation}\prod_{\substack{1\le i\le n\\i\not=j}}(x_j-x_i)=(-1)^{1+j}\frac{\prod_{1\le i<l\le n}(x_i-x_l)}{\prod_{\substack{1\le i<l\le n\\i\not=j, l\not=j}}(x_i-x_l)},\end{equation} we obtain
\begin{equation}\label{3.19}
\sum_{j=1}^n\frac{x_j^{k+n-1}}{\prod_{\substack{1\le i\le n\\i\not=j}}(x_j-x_i)}=\frac{\sum_{j=1}^n(-1)^{1+j}x_j^{k+n-1}\prod_{\substack{1\le i<l\le n\\i\not=j, l\not=j}}(x_i-x_l)}{\det\left(x_i^{n-j}\right)_{1\le i,j\le n}}.\end{equation}
From Definition \ref{def schur}, it follows that $h_k(x_1,\ldots,x_n)=s_{(k)}(x_1,\ldots,x_n)$. The bialternant Cauchy–Jacobi formula \eqref{schur determinant}, applied to the partition $\lambda=(k)$, then gives
\[
\raisebox{5.5ex}{$h_k(x_1,\ldots,x_n)= \ \ \ $}
\frac{
\det\!\begin{bmatrix}
x_1^{k+n-1} & x_2^{k+n-1} & \cdots & x_n^{k+n-1}\\
x_1^{n-2} & x_2^{n-2} & \cdots & x_n^{n-2}\\
\vdots & \vdots & \ddots & \vdots\\
1 & 1 & \cdots & 1
\end{bmatrix}
}{
\det\left(x_i^{n-j}\right)_{1\le i,j\le n}
}.
\]
Therefore, by applying the Laplace expansion of the determinant along the first row, we see that
\begin{equation}h_k(x_1,\ldots,x_n)=\frac{\sum_{j=1}^n(-1)^{1+j}x_j^{k+n-1}\prod_{\substack{1\le i<l\le n\\i\not=j, l\not=j}}(x_i-x_l)}{\det\left(x_i^{n-j}\right)_{1\le i,j\le n}}.\end{equation}
We combine this with \eqref{3.19} and \eqref{3.18} to obtain \eqref{complete homo contour}. We now apply the Jacobi-Trudi formula, Theorem \ref{JacobiTrudi}, to give the result.
\end{proof}
The methods used in \cite{kn:crs06,kn:keawei1,kn:acrs} to calculate joint moments all relied on differentiating the following lemma, which is due to \cite{kn:crs06}. We use Lemma \ref{contour JC} to give a new proof.
\begin{lemma}[Lemma 3 of \cite{kn:crs06}]\label{contour products} Let $z(x)=\frac{1}{1-e^{-x}}$. Then, for $\alpha_1,\ldots , \alpha_{2k}\in\mb{C}$, we have
    \[
    \begin{split}
&\int_{\mathrm{U}(N)} \prod_{j=1}^k \Lambda_X(e^{-\alpha_j}) \Lambda_{X^\dagger}(e^{\alpha_{j+k}}) \, \mathrm{d}U\\ &= \frac{1}{k!(2\pi \mf{i})^k}
\int_{|w_i|=1}
e^{N\sum_{j=1}^k (w_j-\alpha_j)}
\prod_{\substack{1 \le i \le k \\ 1 \le j \le 2k}}
z(w_i - \alpha_j)
\prod_{i \ne j}
z(w_i - w_j)^{-1}
\prod_{j=1}^k \mathrm{d}w_j .
\end{split}\]
\end{lemma}
\begin{proof}
 By Lemma \ref{BG lemma}, 
    \[\int_{\mathrm{U}(N)} \prod_{j=1}^k \Lambda_X(e^{-\alpha_j}) \Lambda_{X^\dagger}(e^{\alpha_{j+k}}) \, \mathrm{d}U=e^{-N\sum_{j=1}^k\alpha_j}s_{\langle N^k\rangle}(e^{\alpha_1},\ldots ,e^{\alpha_{2k}}).\]
    Applying Lemma \ref{contour JC} followed by Andréief's identity gives
      \[
    \begin{split}
&\int_{\mathrm{U}(N)} \prod_{j=1}^k \Lambda_X(e^{-\alpha_j}) \Lambda_{X^\dagger}(e^{\alpha_{j+k}}) \, \mathrm{d}U\\ &= \frac{1}{k!(2\pi \mf{i})^k}
\int_{|w_i|=1}
e^{N\sum_{j=1}^k (w_j-\alpha_j)}
\prod_{\substack{1 \le i \le k \\ 1 \le j \le 2k}}
z(w_i - \alpha_j)
\det_{k\times k}\left(e^{(i-1)w_j}\right)\det_{k\times k}\left(e^{-(i-1)w_j}\right)
\prod_{j=1}^k \mathrm{d}w_j .
\end{split}\]
It is then easy to verify the identity:
\[\det_{k\times k}\left(e^{(i-1)w_j}\right)\det_{k\times k}\left(e^{-(i-1)w_j}\right)=\prod_{i \ne j}
z(w_i - w_j)^{-1}\]
\end{proof}

\subsubsection{A short proof of a result of Alvarez, Conrey, Snaith and Rubinstein}
In this subsection, we will give an alternative proof of the following theorem due to Alvarez, Conrey, Snaith and Rubinstein \cite{kn:acrs}. Our strategy is to use the Jacobi-Trudi formula and a Taylor expansion of the complete homogeneous symmetric polynomials around the identity. The purpose of reestablishing this result is to emphasise that the joint moment problem can be understood through combinatorics and without the introduction of contour integrals.
\begin{theorem}[Theorem 2.1 of \cite{kn:acrs}]\label{ACRS theorem1}
 For $k$ a non-negative integer, and $x \in \C$ with $|x|=1$,
\begin{eqnarray}
   \label{eq ACRS}
   &&  \int_{\mathrm{U}(N)} \abs{\Lambda_U'(x)}^{2k} \mathrm{d}U \nonumber \\
    && = (-1)^{\binom{k}{2}}\sum_{r=0}^k \binom{k}{r}N^{k-r} (-1)^r\sum_{\substack{\sum_{i=1}^k m_{i} = k+r}}\binom{k+r}{m_{1}, \dots, m_{k}}    \nonumber  
    \det_{k\times k}\left[ \binom{N+k+i+j-2}{2k+m_j -1} \right].
    \end{eqnarray}
\end{theorem}
\begin{proof}
    Using Theorem \ref{Theorem 1} with $\alpha=\langle1^k\rangle$, which expresses $|\Lambda|_N(k;2\!\cdot\!\alpha)$ in terms of $\mathcal{M}_N$, we have
\begin{equation}
    \int_{\mathrm{U}(N)} \abs{\Lambda_U'(x)}^{2k} \mathrm{d}U=\sum_{r=0}^k\binom{k}{r}N^{k-r}(-1)^r\mathcal{M}_N(k,k+r).
\end{equation}
     We may apply Proposition \ref{prop modified moments} to express $\mathcal{M}_N$ in terms of the Taylor coefficients of Schur polynomials:
        \begin{align}
\mathcal{M}_N(k,k+r)&=\partial_{x_1}\cdots\partial_{x_{k+r}}s_{\langle N^k\rangle}(1+x_1,\ldots ,1+x_{2k})\big|_{x_1=\cdots =x_{2k}=0}\\
        &=\partial_{x_1}\cdots\partial_{x_{k+r}}\det_{k\times k}\left(h_{N-i+j}(1+x_1,\ldots ,1+x_{2k})\right)\big|_{x_1=\cdots =x_{2k}=0},\label{4.31}
        \end{align}
        where the second equality follows from the Jacobi-Trudi formula.
        When differentiating a determinant with respect to a variable $t$, by the product rule we have
        \[\frac{\mathrm{d}}{\mathrm{d}t}\det(a_1,\ldots ,a_k)=\det(\frac{\mathrm{d}}{\mathrm{d}t}a_1,\ldots ,a_k)+\cdots+\det(a_1,\ldots ,\frac{\mathrm{d}}{\mathrm{d}t}a_k),\]
        where the $a_i$ are column vectors. We repeat this $k+r$ times to the determinant in \eqref{4.31}. Given that the $h_i$ are symmetric polynomials, we only need to count how many times a given column has been differentiated. That is,
          \begin{equation}
        \begin{split}\label{4.32}
            &\partial_{x_1}\cdots\partial_{x_{k+r}}\det_{k\times k}\left(h_{N-i+j}(1+x_1,\ldots ,1+x_{2k})\right)\big|_{x_1=\cdots =x_{2k}=0}\\
           & =\sum_{\sum_{j=1}^km_j=k+r}\binom{k+r}{m_1,\ldots ,m_k}\det_{k\times k}\left( \partial_{x_1}\cdots\partial_{x_{m_j}}h_{N-i+j}(\{1\}^{2k})\right).
            \end{split}
        \end{equation}
These derivatives can be computed using the Taylor expansion of $h_m(1+x_1,\ldots ,1+x_n)=s_{(m)}(1+x_1,\ldots ,1+x_n)$, which is given by
\begin{equation}\label{complete homo taylor}h_m(1+x_1,\ldots ,1+x_n)=\sum_{l\ge 0}\binom{n+m-1}{m-l}h_l(x_1,\ldots,x_n).\end{equation}
This can be derived easily from the generalised binomial theorem, Theorem \ref{binom}.
        After setting $m=N-i+j$ and $n=2k$, we apply \eqref{complete homo taylor} to obtain
        \begin{equation}\label{4.33}
            \partial_{x_1}\cdots\partial_{x_{m_j}}h_{N-i+j}(\{1\}^{2k})=\binom{N+2k-i+j-1}{2k+m_j-1}.
        \end{equation}
        Combining \eqref{4.31}, \eqref{4.32} and \eqref{4.33} gives
        \begin{equation}
    \mathcal{M}_N(k,k+r)=\sum_{\sum_{j=1}^km_j=k+r}\binom{k+r}{m_1,\ldots ,m_k}\det_{k\times k}\left(\binom{N+2k-i+j-1}{2k+m_j-1}\right).
        \end{equation}
       After reordering the rows of this determinant, we pick up a factor $(-1)^{\binom{k}{2}}$ and the result follows.
        
\end{proof}

\appendix 
\section{Appendix}
\begin{proposition}\label{appendix prop}
     Let $h,k\in\mb{N}_0$ with $h\le k$ and set $\alpha=\langle2^{h}\rangle$. Then,
    \begin{equation}
\begin{split}
|\mathcal{Z}|(k;2\!\cdot\!\alpha) = \ &  \sum_{\substack{l_0 + \dots + l_k = 2h \\ l_i\ge0}} \sum_{\substack{t_0 + \dots + t_k = 2l_0 \\ t_i\ge0}} (-2)^{-t_0} \binom{2h}{l_0, \dots, l_k} \binom{2l_0}{t_0, \dots, t_k} \\
& \times \left( \prod_{i=1}^{k} \frac{1}{(2k - i + 2l_i + t_i)!} \right) \left( \prod_{1 \le i < j \le k} (2l_i - 2l_j + t_i - t_j - i + j) \right).
\end{split}
\end{equation}
\end{proposition}
\begin{proof}
    By \eqref{4.8}, and using the fact that $|\alpha|$ is even, we have
    \begin{equation}\label{4.11}
        |\mathcal{Z}|(k;2\!\cdot\!\alpha)=(-1)^{|\alpha|}\sum_{\substack{\mf{m}=(m_1,\ldots,m_k)\\m_i\ge0}}C_{{\mf{m}}}\frac{\prod_{1\le i<j\le k}(m_i-m_j+j-i)}{\prod_{i=1}^k(m_i+2k-i)!},
    \end{equation}
    where $C_{{\mf{m}}}\equiv C_{m_1,m_2,\ldots ,m_k}$ is defined as 
    \begin{equation}\label{Cm}
    C_{\mf{m}}=\sum_{\substack{\beta\leq 2\cdot\alpha\\ }}\left(\frac{-1}{2}\right)^{2|\alpha|-|\beta|}\frac{(\alpha!)^2}{(2\!\cdot\!\alpha-\beta)!}N_{\beta,{m}}.\end{equation}
    By Lemma \ref{Kostkalemma}, we have 
    \[N_{\beta,\mf{m}}=[x_1^{m_1}\cdots x_k^{m_k}]h_{\beta}(x_1,\ldots,x_k),\]
where $[x_1^{m_1}\cdots x_k^{m_k}]h_{\beta}(x_1,\ldots,x_k)$ denotes the coefficient of $x_1^{m_1}\cdots x_k^{m_k}$ in the monomial expansion of $h_{\beta}(x_1,\ldots,x_k)$. Furthermore, 
it is a simple check to see that
\[\frac{1}{4}-h_{1}(x_1,\ldots,x_k)+2h_{2}(x_1,\ldots,x_k)=\sum_{l\le2}\left(-\frac{1}{2}\right)^{2-l}\frac{2!}{(2-l)!}h_l(x_1,\ldots,x_k).\]
Therefore, using \eqref{Cm} with $\alpha=\langle2^h\rangle$, we have
        \begin{equation}
        \begin{split}
        C_{\mf{m}}&=[x_1^{m_1}\cdots x_k^{m_k}]\prod_{i=1}^{2h}\sum_{\beta_i\le 2\cdot\alpha_i}\left(-\frac{1}{2}\right)^{2\cdot\alpha_i-\beta_i}\frac{2\!\cdot\!\alpha_i!}{(2\!\cdot\!\alpha_i-\beta_i)!}h_{\beta_i}(x_1,\ldots,x_k)\\
       & =
       [x_1^{m_1}\cdots x_k^{m_k}] \left(\frac{1}{4}-h_1(x_1,\ldots,x_k)+2h_2(x_1,\ldots,x_k)\right)^{2h},
        \end{split}
    \end{equation}
    where we have used the fact that $2\!\cdot\!\alpha_i=2$ for all $i\le 2h$ in the second equality.
   From the definition of the complete homogeneous polynomials, we have
    \begin{align}
        &h_1(x_1,\ldots ,x_k)=\sum_{i=1}^k x_i\quad \text{and} \quad 2h_2(x_1,\ldots ,x_k)=\left(\sum_{i=1}^kx_i\right)^2+\sum_{i=1}^k x_i^2
    \end{align}
    Therefore,
    \begin{equation}
        \left(\frac{1}{4}-h_1(x_1,\ldots ,x_k)+2h_2(x_1,\ldots ,x_k)\right)^{2h}=\left(\left(\frac{1}{2}-\sum_{i=1}^k x_i\right)^2 +\sum_{i=1}^k x_i^2\right)^{2h}.
    \end{equation}
    Expanding this expression into a sum over multinomials gives
    \begin{equation}
        \begin{split}
             &\left(\left(\frac{1}{2}-\sum_{i=1}^k x_i\right)^2 +\sum_{i=1}^kx_i^2\right)^{2h}\\
             &=\sum_{\substack{l_0 + \dots + l_k = 2h \\ }} \sum_{\substack{t_0 + \dots + t_k = 2l_0 \\ }} (-2)^{-t_0} \binom{2h}{l_0, \dots, l_k} \binom{2l_0}{t_0, \dots, t_k} \prod_{i=1}^kx_i^{2l_i+t_i},
        \end{split}
    \end{equation}
    and thus, we obtain
    \begin{equation}
        C_{\mf{m}}=\sum_{\substack{l_0 + \dots + l_k = 2h \\ }} \sum_{\substack{t_0 + \dots + t_k = 2l_0 \\ 2l_i+t_i=m_i \ (1\le i\le k)}} (-2)^{-t_0} \binom{2h}{l_0, \dots, l_k} \binom{2l_0}{t_0, \dots, t_k}.
    \end{equation}
    Combining this with \eqref{4.11} and making the substitution $m_i=2l_i+t_i$ gives the result.
\end{proof}
\bibliographystyle{alpha}
\bibliography{bibliography}
\end{document}